\documentclass[10pt]{article}
\usepackage[utf8]{inputenc}
\usepackage[english]{babel} 
\usepackage{enumitem}
\usepackage{amssymb}
\usepackage{amsmath}
\usepackage{amsthm}
\usepackage{txfonts}
\usepackage{mathdots}
\usepackage{enumitem}
\usepackage{parskip}
\usepackage[a4paper, total={6in, 8in}]{geometry}
\usepackage[classicReIm]{kpfonts}
\usepackage[dvipsnames]{xcolor}
\usepackage[dvips]{graphicx}
\newtheorem{theorem}{Theorem}[section]
\newtheorem{cor}[theorem]{Corollary}
\newtheorem{lem}[theorem]{Lemma}

 \newtheorem{rmk}[theorem]{Remark}

  \newtheorem{example}[theorem]{Example}

\newcommand{\R}{\mathbb{R}}

\newcommand{\C}{\mathbb{C}}
\newcommand{\N}{\mathbb{N}}

\newcommand{\subscript}[2]{$#1 _ #2$}

\title{Ground state solutions for quasilinear Schr\"{o}dinger type equation involving anisotropic p-laplacian}
\author{ Kaushik Bal and Sanjit Biswas }
\date{\today}
\begin{document}
\maketitle
\section*{Abstract}
 This paper is concerned with the existence of a nonnegative ground state solution of the following quasilinear Schr\"{o}dinger equation
\begin{align*}
\begin{cases}
     -\Delta_{H,p}u+V(x)|u|^{p-2}u-\Delta_{H,p}(|u|^{2\alpha}) |u|^{2\alpha-2}u=\lambda |u|^{q-1}u \text{ in }\R^n\\
     u\in W^{1,p}(\R^n)\cap L^\infty(\R^N)
\end{cases}
\end{align*}
where $N\geq2$; $(\alpha,p)\in D_N=\{(x,y)\in \R^2 : 2xy\geq y+1,\; y\geq2x,\; y<N\}$ and $\lambda>0$ is a parameter. The operator $\Delta_{H,p}$ is the reversible Finsler p-Laplacian operator with the function $H$ being the Minkowski norm on $\R^N$. Under certain conditions on $V$, we establish the existence of a non-trivial non-negative bounded ground state solution of the above equation.

\section*{Introduction and main results}\label{int}
In this paper, we are concerned with the following problem
\begin{align}\label{maineq}
   \begin{cases}
     -\Delta_{H,p}u+V(x)|u|^{p-2}u-\Delta_{H,p}(|u|^{2\alpha}) |u|^{2\alpha-2}u=\lambda |u|^{q-1}u \text{ in } \R^n\\
     u\in W^{1,p}(\R^n)\cap L^\infty(\R^N)
   \end{cases}
\end{align}
where $\N\geq 2$, $(\alpha,p)\in D_N=\{(x,y)\in \R^2 : 2xy\geq y+1,\; y\geq2x,\; y<N\}$ and  $\lambda>0$ is a parameter. The operator $\Delta_{H,p}$ is defined as $$\Delta_{H,p}u :=\mbox{div}(H(Du)^{p-1}\nabla_\eta H(Du))$$
known as anisotropic p-Laplacian, where $\nabla_\eta $ denotes the gradient operator with respect to $\eta$ variable. The function $H:\R^N\to [0, \infty)$ is a Minkowski norm satisfying the following properties:
\begin{enumerate}[label=(H\arabic*)]
    \item $H$ is a norm on $\R^N$;
    \item $H\in C^4(\R^N\setminus\{0\})$;
    \item The Hessian matrix $\nabla_\eta^2(\frac{H^2}{2})$ is positive definite in $R^N\setminus\{0\}$;
    \item $H$ is uniformly elliptic, that means the set $$\mathscr{B}_1^H:=\{ \xi\in\R^N : H(\xi)<1 \}$$ is uniformly convex, i.e. there exists $\Lambda>0$ such that $$\langle D^2H(\xi)\eta, \eta\rangle\geq \Lambda|\eta|^2, \text{ } \forall\xi\in\partial\mathscr{B}_1^H, \forall \eta\in\nabla H(\xi)^\perp.$$
    \item There exists a positive constant $M=M(N,H)$ such that for $1\leq i,j\leq N$,
    $$HH_{x_ix_j}+H_{x_i}H_{x_j}\leq M$$
\end{enumerate}
For more details about $H$, one may consult \cite{KB, FC, CX} and the references therein. A few examples of $H$ are as follows:
\begin{example}
    For $k=2$ or $k\geq 4$, we define $H_k:\R^N\to \R$ as
    \begin{align}
        H_k(x_1,x_2,...x_N):=(\sum_{i=1}^{N} |x_i|^k)^\frac{1}{k}.
    \end{align}
\end{example}
\begin{example}(\text{Mezei-Vas} \cite[\text{ Remark 2.2}]{VM})
    For $\rho,\mu>0$, we define $H_{\rho,\mu}:\R^N\to\R$ as 
    \begin{align}
        H_{\rho,\mu}(x_1,x_2,...,x_N):=\sqrt{\rho\sqrt{\sum_{i=1}^{N}x_i^4}+\mu\sum_{i=1}^{N}x_i^2}
    \end{align}
\end{example}
\begin{rmk}\label{equiv}
    There exists two constants $A,B>0$ such that
   \begin{align}
       A\;||x||\leq H(x)\leq B\;||x||, \mbox{ for all} \;x\in\R^N.
   \end{align}
\end{rmk}
\begin{rmk}
    If $H=H_k$ then $$\Delta_{H,p}u=\begin{cases}
        \Delta_pu, \text{ if } k=2\;\mbox{and} \;1<p<\infty\\
        \sum_{i=1}^{N}\frac{\partial}{\partial x_i}(|\frac{\partial u}{\partial x_i}|^{p-2}\frac{\partial u}{\partial x_i}), \text{ if } k=p\in (1,\infty)
    \end{cases}$$
\end{rmk}
 Throughout this paper, we will assume that the potential $V\in C^3(\R^N,\R)$ satisfies the following conditions:
\begin{enumerate}[label=(\subscript{v}{{\alph*}})]
    \item  $0\leq V(x)\leq V(\infty):=\liminf_{|x|\to\infty} V(x)\infty$ and $V$ is not identically equal to $V(\infty)$.
    \item $\langle \nabla V(x), x\rangle \in L^\infty(\R^N)\cup L^{\frac{N}{p}}(\R^N)$ and $NV(x)+\langle \nabla V(x), x\rangle\geq 0$.
\end{enumerate}

The equation (\ref{maineq}) arises in various branches of mathematical physics. For instance, when $H=H_2,\; p=2$ and $\alpha=1$, solutions of (\ref{maineq}) are standing wave solution of the quasilinear Schr\"{o}dinger equation of the form
\begin{align}\label{phy}
    \iota\partial_t\Phi+\Delta\Phi+k\Delta h(|\Phi|^2)h'(|\Phi|^2)\Phi+\rho(|\Phi|^2)\Phi-V\Phi=0
\end{align}
where $V:\R^N\to\R$ is a given potential, $\Phi:\R\times\R^N\to \C$; $h,\rho :\R^+\to \R$ are functions and $k$ is real constant. It is worth mentioning that the semilinear case corresponding to $k=0$ has been extensively studied by many authors ( see \cite{HP, AK, KJ} and the references therein).\\
The general equation (\ref{phy}) with various forms of $h$ has been derived as models of several physical phenomena. 
\begin{enumerate}[label=(\alph*)]
    \item The superfluid film equation in plasma physics has the structure (\ref{phy}) for $h(s)=s$, see \cite{Kur}.
\item For $h(s)=(1+s)^\frac{1}{2}$, equation (\ref{phy}) models the self-channeling of a high-power ultrashort laser in matter (see \cite{BG, BR}).
\end{enumerate}
In recent years, extensive studies have focused on the existence of solutions for quasilinear Schr\"{o}dinger equation of the form  
\begin{align}\label{phy1}
    -\Delta u+V(x)u-k\Delta(u^2)u=g(u)\; \text{in}\; \R^N
\end{align}
where $k>0$ is a constant. The existence of a nonnegative solution for (\ref{phy1}) was proved for $N=1$ and $g(u)=|u|^{p-1}u$ by Poppenberg et al. \cite{PS} and for $N\geq 2$ by Wang et al. \cite{Wang1}. In \cite{Wang2}, Wang and Liu have proved that the equation (\ref{phy1}) for $k=\frac{1}{2}$ and $g(u)=\lambda |u|^{p-1}u$ has a positive ground state solution when $3\leq p<2.2^*-1$ and the potential $V\in C(\R^N,\R)$ satisfies one of the following conditions:
\begin{enumerate}[label=$(v_\alph*)$,start=3]
\item $\lim_{|x|\to\infty} V(x)=\infty.$
\item $V(x)=V(|x|)$ and $N\geq 2.$
\item $V$ is periodic in each variable.
\item $V_\infty:=\lim_{|x|\to\infty}=||V||_{L^\infty(\R^N)}<\infty.$
\end{enumerate}
They also have proved in \cite{Liu} the existence of both one-sign and nodal ground state of soliton type solutions for (\ref{phy1}) when $3\leq p<22^*-1$ and the potential $V\in C(\R^N,\R)$ satisfies 
\begin{enumerate}[label=$(v_\alph*)$,start=7]
\item $0<\inf_{\R^N} V(x)\leq V(\infty):=\lim_{|x|\to\infty} V(x)<\infty.$
\item There are positive constants $M,\;A$ and $m$ such that for $|x|\geq M$, $V(x)\leq V_\infty-\frac{A}{1+|x|^m}.$
\end{enumerate}
Similar work with critical growth has been done in \cite{QW}. Ruiz and Siciliano have proved the existence of a ground-state solution for (\ref{phy1}) with $g(u)=|u|^{p-1}u$, $N\geq 3$, $3\leq p<22^*-1$ under the following assumptions:
\begin{enumerate}[label=$(v_\alph*)$,start=9]
\item $0<V_0\leq V(x)\leq V(\infty):=\lim_{|x|\to\infty} V(x)>\infty$ and $\langle \nabla V(x), x\rangle\in L^\infty(\R^N).$
\item For every $x\in \R^N$, the following map is concave $$s\to s^\frac{N+2}{N+p+1}V(s^\frac{1}{N+p+1}x)$$
\end{enumerate}
In \cite{Chen}, Chen and Xu have proved that the equation (\ref{phy1}) for $g(u)=\lambda |u|^{p-1}u$ has a positive ground state solution for large $\lambda>0$ under the condition $N\geq 3$, $3\leq p<22^*-1$, and the following assumptions on $V\in C(\R^N,\R)$,
\begin{enumerate}[label=$(v_\alph*)$,start=11]
\item $0\leq V(x)\leq V(\infty):=\liminf_{|x|\to\infty} V(x)\infty$ and $V$ is not identically equal to $V(\infty)$.
\item $\langle \nabla V(x), x\rangle \in L^\infty(\R^N)\cup L^{\frac{N}{2}}(\R^N)$ and $NV(x)+\langle \nabla V(x), x\rangle\geq 0$.
\end{enumerate}
We end the literature review by mentioning Chen and Zhang \cite{ZJ}, who proved the existence of a positive ground-state solution of 
\begin{align}
    -\Delta u+V(x)u-k\Delta(u^2)u= A(x)|u|^{p-1}u+\lambda B(x)|u|^{22^*-1}
\end{align}
when $N\geq 3, 2\leq p<22^*-1$ and under the following assumptions:
\begin{enumerate}[label=$(v_\alph*)$,start=13]

\item $V\in C^1(\R^N,\R^+), 0<V_0:=\inf_{x\in\R^N}V(x)\leq V(x)\leq V_\infty:=\lim_{|x|\to\infty}<\infty$ and $V(x)\nequiv V_\infty$;
\item $\langle \nabla V(x),x\rangle\in L^\infty(\R^N)$, $\langle \nabla V(x),x\rangle\leq 0$; 
\begin{itemize}
    \item $A\in C^1(\R^N,\R^+),\; \lim_{|x|\to\infty}A(x)=A_\infty\in (0,\infty), A(x)\geq A_\infty, 0\leq \langle \nabla A(x),\;x\rangle\in L^\infty(\R^N)$
    \item $B\in C^1(\R^N,\R^+), \;\lim_{|x|\to\infty} B(x)=B_\infty\in (0,\infty), B(x)\geq B_\infty, 0\leq \langle \nabla B(x),\;x\rangle\in L^\infty(\R^N)$
\end{itemize}
\end{enumerate}
Before stating our main results we define two sets
\begin{equation}
    \begin{split}
        \Pi=(p-1, 2\alpha p^*-2\alpha q+p-1)\cup (2\alpha p-2\alpha, 2\alpha p^*-2\alpha)\cup (p+2\alpha-2, 2\alpha p^*-2\alpha q+p+2\alpha-2)\\
\cup (\frac{p-1}{2\alpha}+2\alpha-1, p^*-p+\frac{p-1}{2\alpha}+2\alpha-1)\cup (2\alpha p-1, 2\alpha p^*-1).
    \end{split}
\end{equation}
and, $$D^N:=\{(x,y)\in\R^2: 2xy\geq y+1, y\geq 2x+1, y<N\}.$$

Our main results are as follows: 
\begin{theorem}\label{TC}
   Let $V$ be a constant potential, $N\geq 2$, $(\alpha, p)\in D_N$ and $q\in \Pi\cap (p-1, 2\alpha p^*-1)$. Then for large $\lambda>0$, equation (\ref{maineq}) admits a  non-trivial non-negative bounded ground state solution $u\in C^1(\R^N)$.
\end{theorem}
\begin{theorem}\label{TV}
    Suppose that the potential $V$ satisfies $(v_1)-(v_2)$ and $\N\geq 3$. We also assume $p,\alpha$ satisfy one of the following
    \begin{itemize}
        \item $p=2$, $(\alpha, p)\in D_N$;
        \item $(\alpha, p)\in D^N$
    \end{itemize}
    and $2\alpha p-1\leq q< 2\alpha p^*-1$ (one has to take strict inequality if $2\alpha p-1\notin \Pi$). Then for large $\lambda>0$, equation (\ref{maineq}) admits a non-trivial non-negative bounded ground state solution $u\in C^1(\R^N)$.
\end{theorem}
\begin{rmk}
If $H=H_2$, $p=2$, and $\alpha=1$ then by using the strong maximum principle for Laplacian in Theorem \ref{TC} and Theorem \ref{TV}, we obtain a positive ground state solution of (\ref{maineq}), which generalize the main results of Chen et al. \cite{Chen}.
\end{rmk}
The paper is organized as follows. In section \ref{Pre}, we reformulate this problem in an appropriate Orlicz space and discuss a few useful lemmas. In section \ref{Aux}, we prove some auxiliary results. Section \ref{TCP} is devoted to the proof of Theorem \ref{TC}. Finally, in section \ref{TVP}, we will give the proof of Theorem \ref{TV}.  \\
\textbf{Notation:} In this work we will use the following notations:
\begin{itemize}
    \item $C$ represents a positive constant whose value may change from line to line.
    \item $W^{1,p}(\R^N):=\{ u\in L^p(\R^N) : \nabla u\in L^p(\R^N)\}$ with the usual norm $$||u||_{1,p,\R^N}^p=\int_{\R^N} [|u|^p+|\nabla u|^p] dx$$
    \item $D^{1,p}(\R^N):=\{ u\in L^{p^*}(\R^N) : \nabla u\in L^p(\R^N)\}$ with the norm $$||u||^p:=\int_{\R^N} |\nabla u|^p dx,$$ where $p^*=\frac{Np}{N-p}$.
    \item For a function $h\in L^1_{loc}(\R^N)$, we denote $$\int h(x) dx:=\int_{\R^N} h(x) dx.$$
    \item $C_c^\infty(\R^N):=\{ u\in C^\infty(\R^N) |\; u \text{ has compact support.}\}$
    \item $X'$ denotes the dual of $X$ and $\langle\cdot , \cdot \rangle$ denotes the duality relation.
    \item For $\Omega\subset\R^N$,\;$|\Omega|$ denotes the Lebesgue measure of $\Omega$.
    \item o(1) represents a quantity which tends to 0 as $n\to\infty$.
    \item The symbols $\rightharpoonup$ and $\rightarrow$ denote weak convergence and strong convergence respectively.
\end{itemize}
\section{ Variational Framework and Preliminaries}\label{Pre}
The variational form of the equation (\ref{maineq}) is 
 \begin{align*}
     I(u)=\frac{1}{p}\int [1+(2\alpha)^{p-1}|u|^{(2\alpha-1)p}]H(Du)^p+ V(x)|u|^p]dx-\frac{\lambda}{q+1}\int |u|^{q+1} dx
 \end{align*}
which is not well-defined on $W^{1,p}(\R^N)$. Inspired by  \cite{Wang2, Wang1, Chen}, we choose a transformation $u=f(v)$ where $f$ is defined as follows:
\begin{align}\label{DE}
  \begin{cases}
     f'(t)=[1+(2\alpha)^{p-1}|f(t)|^{(2\alpha-1)p}]^{-\frac{1}{p}},\; t>0\\
    f(0)=0 \text{ and }f(-t)=-f(t) \text{ for all $t\in\R$ }
\end{cases}
\end{align}
Under the transformation $u=f(v)$, the above functional becomes
\begin{align}\label{TF}
     I(f(v))=\frac{1}{p}\int [H(Dv)^p+ V(x)|f(v)|^p]dx-\frac{\lambda}{q+1}\int |f(v)|^{q+1} dx.
 \end{align}
 We give some important properties of $f$, which will be useful for establishing our main results.
\begin{lem}\label{P}
The function $f$ enjoys the following properties:
\begin{enumerate}[label=(\roman*)]
    \item $f$ is uniquely defined, a $C^2-$function, and invertible.
    \item $|f(t)|\leq |t|$ for all $t\in\R$.
    \item $\frac{f(t)}{t}\to 1$ as $t\to 0$.
    \item There exists $a>0$ such that $f(t)t^{-\frac{1}{2\alpha}}\to a$ as $t\to\infty$.
    \item $|f(t)|\leq (2\alpha)^\frac{1}{2\alpha p}|t|^\frac{1}{2\alpha}$ for all $t\in\R^.$
    \item $f(t)\leq 2\alpha t f'(t)\leq2\alpha f(t)$ for all $t\geq 0.$
    \item There exists $C>0$ such that
         $$f(t)\geq \begin{cases}
             C|t|,\; \text{ if }\; |t|\leq 1\\ C|t|^\frac{1}{2\alpha},\; \text{ if }\; |t|\geq 1
         \end{cases}$$
         \item  $|f|^p$ is convex if and only if $p\geq 2\alpha.$
    \item $|f^{2\alpha-1}(t)f'(t)|\leq (\frac{1}{2\alpha})^\frac{p-1}{p}$, for all $t\in\R.$
    \item  There exist two positive constants $M_1$ and $M_2$ such that $$|t|\leq M_1|f(t)|+M_2|f(t)|^{2\alpha}, \text{ for all $t\in \R$}.$$
\end{enumerate}
\end{lem}
\begin{proof}
    The proof of the first seven properties can be found in \cite{FG}. To prove property $(viii)$ we define $\phi:\R\to\R$ by $\phi(t)=|f(t)|^p$. It is easy to check that, 
    \begin{enumerate}[label=\roman*)]
        \item $\phi'(t)=p|f(t)|^{p-2}f(t)f'(t)$ 
        \item $\phi"=p|f|^{p-2}[ff"+(p-1)f'^2].$
    \end{enumerate}
Now, $\phi$ is convex if and only if 
\begin{align}
    ff"+(p-1)f'^2=\frac{(p-1)+(p-2\alpha)(2\alpha)^{p-1}|f|^{p(2\alpha-1)}}{[1+(2\alpha)^{p-1}|f|^{p(2\alpha-1)}]^{1+\frac{2}{p}}}\geq 0.
\end{align}
Moreover, $ff"+(p-1)f'^2\geq 0$ if and only if $(p-1)+(p-2\alpha)(2\alpha)^{p-1}|f|^{p(2\alpha-1)}\geq 0$. Hence, we can conclude that if $p\geq 2\alpha$ then $\phi$ is convex. Conversely, if $\phi$ is convex then $$(p-1)+(p-2\alpha)(2\alpha)^{p-1}|f(t)|^{p(2\alpha-1)}\geq 0 , \text{ for all } t>0.$$ Therefore,
\begin{align*}
\frac{(p-1)+(p-2\alpha)(2\alpha)^{p-1}|f|^{p(2\alpha-1)}}{t^\frac{p(2\alpha-1)}{2\alpha}}\geq 0.
\end{align*}
Take $t\to \infty$, we have
\begin{align*}
    \lim_{t\to \infty}[(p-1)t^{-\frac{p(2\alpha-1)}{2\alpha}}+(p-2\alpha)(2\alpha)^{p-1}(|f(t)|{t^{-\frac{1}{2\alpha}}})^{p(2\alpha-1)}]\geq 0
\end{align*}
Using the fact $2\alpha\geq \frac{p+1}{p}$ and the property (iv), we have 
$(p-2\alpha)(2\alpha)^{p-1}a^{p(2\alpha-1)}\geq 0$ that is, $p\geq 2\alpha$.\\
The definition of $f'$ follows the property (ix) and property (x) is an immediate consequence of properties (iv) and (v).  
\end{proof}
Now, we are going to define a suitable space so that RHS of (\ref{TF}) makes sense. We define a normed space $X:=\{v\in W^{1,p}(\R^N): \int V(x)|f(v)|^p\;dx< \infty\}$ equipped with the norm 
\begin{align}\label{n1}
    ||v||=||H(Dv)||_{L^p(\R^N)}+\inf_{\eta>0} \frac{1}{\eta}\{1+\int V(x)|f(\eta v)|\; dx\}.
\end{align}
The space $X_r=\{v\in X: \text{ v is radial }\}$ is a subspace of $X$.
\begin{rmk}  The following inequality holds true
\begin{align}\label{n2}
     ||v||\leq 1+||H(Dv)||_{L^p(\R^N)}+\int V(x)|f(v)|\; dx
\end{align}
\end{rmk}

\begin{lem}\label{LL}
\begin{enumerate}[label=\roman*)]
The following holds:
    \item There exists a positive constant $C$ such that for all $v\in X$ $$\frac{\int V(x)|f(v)|^p dx}{1+[\int V(x)|f(v)|^p]^\frac{p-1}{p}}\leq C[||\nabla v||_{L^p(\R^N)}^{p^*}+\inf_{\xi>0}\frac{1}{\xi}(1+\int V(x)|f(\xi v)|^p dx)].$$
    \item If $v_n\rightarrow v$ in $X$ then  $$\int V(x)|f(v_n)-f(v)|^p dx\rightarrow 0$$ and $$\int V(x)||f(v_n)|^p-|f(v)|^p| dx\rightarrow 0.$$
    \item If $\int V(x)|f(v_n-v)|^p dx\rightarrow 0$ then 
$$\inf_{\xi>0}\frac{1}{\xi}[1+\int V(x)|f(\xi(v_n-v))|^p dx]\rightarrow 0.$$
\end{enumerate}
\end{lem}

\begin{proof}
    For $\xi>0$ and $v\in X$, we define $$A_\xi=\{x\in\R^N:\xi|v(x)|\leq 1\}.$$ Now, by using (ii) we can write
    \begin{align*}
        \int V(x)|f(v)|^p dx&=\int_{A_\xi} V(x)|f(v)|^p dx+\int_{A_\xi^c} V(x)|f(v)|^p dx\\
        &=\int_{A_\xi} V(x)|f(v)|^{p-1}|v(x)| dx+\int_{A_\xi^c} V(x)|f(v)|^p dx
    \end{align*}

Using H\"{o}lder inequality, (vii) in lemma \ref{P} and $s^\frac{1}{p}\leq 1+s$ for all $s\geq 0$ we have
\begin{align}\label{LL1}
   \int_{A_\xi} V(x)|f(v)|^{p-1}|v(x)| dx&\leq (\int_{A_\xi} V(x)|v|^p dx)^\frac{1}{p}(\int_{A_\xi} V(x)|f(v)|^p| dx)^\frac{p-1}{p}\nonumber\\
   &\leq (\frac{1}{\xi}\int_{A_\xi} V(x)|\xi v|^p dx)^\frac{1}{p}(\int V(x)|f(v)|^p| dx)^\frac{p-1}{p}\nonumber\\
   &\leq C(\frac{1}{\xi}\int_{A_\xi} V(x)|f(\xi v)|^p dx)^\frac{1}{p}(\int V(x)|f(v)|^p| dx)^\frac{p-1}{p}\nonumber\\
   &\leq C[||\nabla v||_{L^p(\R^N)}^{p^*}+\frac{1}{\xi}(1+\int V(x)|f(\xi v)|^p dx)](\int V(x)|f(v)|^p| dx)^\frac{p-1}{p}
\end{align}
If $\xi\geq 1$ then by using (iv) in lemma \ref{P} we get
\begin{align}\label{LL2}
    \int_{A_\xi^c} V(x)|f(v)|^p dx&\leq C\int_{A_\xi^c} V(x)|v|^\frac{p}{2\alpha} dx\leq C\frac{1}{\xi}\int_{A_\xi^c} V(x)|\xi v|^\frac{p}{2\alpha} dx\leq C\frac{1}{\xi}\int_{A_\xi^c} V(x)|f(\xi v)|^p dx\nonumber\\
    &\leq C[||\nabla v||_{L^p(\R^N)}^{p^*}+\frac{1}{\xi}(1+\int V(x)|f(\xi v)|^p dx)]
\end{align}
If $0<\xi<1$ then by using $(v_1)$, Sobolev inequality and Chebyshev inequality, we deduce
\begin{align}\label{LL3}
\int_{A_\xi^c} V(x)|f(v)|^p dx&\leq C\int_{A_\xi^c} V(x)|v|^p dx
    \leq C\int_{A_\xi^c}|v|^p dx
    \leq C[\int_{A_\xi^c}|v|^{p^*} dx]^\frac{p}{p^*}|A_\xi^c|^{1-\frac{p}{p^*}}\nonumber\\
    &\leq C[\int|v|^{p^*} dx]^\frac{p}{p^*}|A_\xi^c|^{1-\frac{p}{p^*}}\leq C[\int |\nabla v|^p dx][\xi^{p^*}\int_{A_\xi^c}|v|^{p^*}]^\frac{p}{N}\nonumber\\
    &\leq C[\int |\nabla v|^p dx][\int_{\R^N}|v|^{p^*}]^\frac{p}{N}
    \leq C[\int |\nabla v|^p dx]^{1+\frac{p^*}{N}}
    \leq C[\int |\nabla v|^p dx]^\frac{p^*}{p}\nonumber\\
    &\leq C[||\nabla v||_{L^p(\R^N)}^{p^*}+\frac{1}{\xi}(1+\int V(x)|f(\xi v)|^p dx)]
\end{align}
Thus, from (\ref{LL1})-(\ref{LL3}) we can conclude that for all $\xi>0$
$$ \int V(x)|f(v)|^p dx\leq C[||\nabla v||_{L^p(\R^N)}^p+\frac{1}{\xi}(1+\int V(x)|f(\xi v)|^p dx)][1+(\int V(x)|f(v)|^p dx)^\frac{p-1}{p}]$$
 which proves the first property. To prove the second property, if $v_n\rightarrow v$ in $X$ then from the first property we have, 
 $$\int V(x)|f(v_n-v)|^p dx\rightarrow 0 \text{ as } n\rightarrow\infty.$$ There exists a nonnegative function $h\in L^1(\R^N)$ such that up to a subsequence $v_n\rightarrow v$ a.e. in $\R^N$ and $\color{red}{V(x)|f(v_n-v)|^p\leq h}$. Since $|f|^p$ is convex and satisfies $\Delta_2$ condition (see M.M Rao\cite{MR}) so $V(x)|f(v_n)|^p\leq CV(x)[|f(v_n-v)|^p+|f(v)|^p]\leq C[h+V(x)|f(v)|^p]$. Moreover, Fatou's lemma ensures $\int V(x)|f(v)|^p dx<\infty$. Thus, by the Dominated Convergence Theorem, we can conclude 
$$\int V(x)|f(v_n)-f(v)|^p dx\rightarrow 0.$$
and $$\int V(x)||f(v_n)|^p-|f(v)|^p| dx\rightarrow 0.$$
To prove the third part, since $\frac{f(t)}{t}$ is nonincreasing in $(0, \infty)$ so for $\xi>1$ and $v\in X$, we obtain
\begin{align}\label{LL4}
    \frac{1}{\xi}(1+\int V(x)|f(\xi v)|^p dx)\leq \frac{1}{\xi}+\xi^{p-1}\int V(x)|f(v_n-v)|^p dx
\end{align}
For every $\epsilon>0$, we can choose $\xi_0>1$ such that $\frac{1}{\xi_0}<\frac{\epsilon}{2}$. There exists a positive integer $N_0$ such that
$\int V(x)|f(v_n-v)|^p dx<\frac{\epsilon}{2\xi_0^{p-1}}$, for all $n\geq N_0.$
Thus, (\ref{LL4}) yields 
\begin{align*}
   \inf_{\xi>0}\frac{1}{\xi}(1+\int V(x)|f(\xi v)|^p dx)\leq \epsilon, \text{ for all } n\geq N_0
\end{align*}
and the property (3) follows.
\end{proof}
\begin{cor}\label{BC}
    If $u_n\rightarrow 0$ in $X$ if and only if $\int [H(Du_n)^p+V|f(u_n)|^p] dx\rightarrow 0$ as $n\rightarrow \infty$.
\end{cor}
\begin{proof}
  The proof is an immediate consequence of the above lemma.   
\end{proof}

Define $E=\{u\in W^{1,p}(\R^N)| \int V(x)|u|^p\;dx<\infty\}$ equipped with the norm $$||u||^p=\int [|\nabla u|^p+V(x)|u|^p]\; dx$$

\begin{cor}
    The embedding $E\hookrightarrow X$ is continuous.
\end{cor}
\begin{proof}
 Using the second property in lemma \ref{P} we have $$\int V(x)|f(v_n)|^p dx\leq \int V(x)|v_n|^pdx.$$
 Thus, if $v_n\rightarrow 0$ in $E$ then $$\int V(x)|f(v_n)|^p dx\rightarrow 0.$$ Hence, lemma \ref{LL} ensures $v_n\rightarrow 0$ in $X$.
\end{proof}

\begin{lem}\label{Comp}
\begin{enumerate}[label=(\roman*)]
\item  The map $v\rightarrow f(v)$ is continuous from $X$ to $L^s(\R^N)$ for $p\leq s\leq 2\alpha p^*$. Moreover, the map is locally compact for $p\leq s<2\alpha p^*$.
\item  The map $v\rightarrow f(v)$ from $X_r$ to $L^s(\R^N)$ is compact for $p<s<2\alpha p^*$.
\end{enumerate}
\end{lem}

\begin{proof}
 It is clear that under the condition $(v_1)$, the embedding $E\hookrightarrow W^{1,p}(\R^N)$ is continuous. Moreover, if $v\in X$ then $f(v)\in E$. There exists $C>0$ such that for every $v\in X$, 
\begin{align}\label{LL5}
    ||f(v)||_{L^p(\R^N)}\leq C ||f(v)||_E\leq C[\int (|\nabla v|^p+V(x)|f(v)|^p)dx]^\frac{1}{p}.
\end{align}
Using the property (v) and (ix) in lemma \ref{P}, we have
\begin{align}\label{LL6}
\int |f(v)|^{2\alpha p^*} dx\leq [\int |\nabla f^{2\alpha}(v)|^p dx]^\frac{p}{p^*}\leq C[\int |\nabla v|^p dx]^\frac{p}{p^*}
\end{align}
 Using (\ref{LL5}), (\ref{LL6}) and the interpolation inequality, we can conclude $f(v)\in L^s(\R^N)$ for all $s \in [p, 2\alpha p^*]$. Let $v_n\rightarrow v$ in $X$. The property (1) in lemma \ref{LL} ensures
 $$\int V(x)|f(v_n)-f(v)|^p dx\rightarrow 0.$$
Furthermore, $Dv_n\rightarrow Dv$ in $(L^p(\R^N))^N$. For every $1\leq i\leq N$, without loss of generality we can assume that there exists $h_i\in L^p(\R^N)$ such that for almost every $x\in\R^N$, 
\begin{equation}
\begin{split}
    v_n(x)\rightarrow v(x) \;\mbox{ as }\; n\to \infty\\
    \frac{\partial v_n}{\partial x_i}(x)\rightarrow \frac{\partial v_n}{\partial x_i}(x)\; \mbox{ as } \;n\to \infty\\
    |\frac{\partial v_n}{\partial x_i}|,\, |\frac{\partial v}{\partial x_i}|\leq h_i
\end{split}
\end{equation}
 By the Dominated Convergence Theorem, we have
$$\int |\frac{\partial}{\partial x_i} (f(v_n))-\frac{\partial}{\partial x_i} (f(v))|^p dx=\int |f'(v_n)\frac{\partial v_n}{\partial x_i}-f'(v)\frac{\partial v}{\partial x_i}|^p dx\rightarrow 0.$$
Therefore, $Df(v_n)\rightarrow Df(v)$ in $L^p(\R^N)$. Consequently, $f(v_n)\to f(v)$ in $E$. Since for all $s\in [p, p^*]$,  $$E\hookrightarrow W^{1,p}(\R^N)\hookrightarrow L^s(\R^N)$$ so $f(v_n)\to f(v)$ in $L^s(\R^N).$ Interpolation inequality and Rellich's lemma complete the first part. The second part is easily deduced from Theorem \ref{ST}. 
\end{proof}
\begin{lem}
    $(X, ||\cdot||)$ is a Banach space.
\end{lem}
\begin{proof}
Let $\{u_n\}$ is a Cauchy sequence in $X$. Since $X\hookrightarrow D^{1,p}(\R^N)$ so there exists $u\in D^{1,p}(\R^N)$ such that $u_n\to u$ in $D^{1,p}(\R^N)$.
By the inequality (1) in lemma \ref{LL}, we observe 
$$ \int V|f(u_n-u_m)|^p dx\to 0 \text{ as } m,n\to\infty.$$
Under the assumption $(v_1)$, we have  
\begin{equation}\label{f2}
         \int |f(u_n-u_m)|^p dx\to 0 \text{ as } m,n\to\infty
\end{equation}
Using (\ref{LL6}), (\ref{f2}) and Interpolation inequality, we get
\begin{equation}\label{f3}
    \int |f(u_n-u_m)|^{2\alpha p} dx\to 0 \text{ as } m,n\to\infty.
\end{equation}
Using property (x) in lemma \ref{P}, we have 
\begin{equation}
    \int |u_n-u_m|^p dx\leq M_1\int |f(u_n-u_m)|^p dx + M_2\int |f(u_n-u_m)|^{2\alpha p} dx\to 0 \text{ as } m,n\to\infty
\end{equation}
which implies $\{u_n\}$ is Cauchy in $L^p(\R^N)$. Completeness property allows us to assume the existence of $w\in L^p(\R^N)$ such that 
\begin{equation*}
\begin{split}
    u_n\to w \text{ in } L^p(\R^N),\;
    u_n\to w \text{ a.e. in } \R^N.
    \end{split}
\end{equation*}
Since $Du_n\to Du$ in $L^p(\R^N)$ so, $Dw=Du$. Consequently, $w\in W^{1,p}(\R^N).$ \\ Our next claim is $$u_n\to w \text{ in } X.$$ 
For every $\epsilon>0$ there exists $N_0\in \N$ such that
$$\int V|f(u_n-u_m)|^p dx <\epsilon \text{ for all } m,n\geq N_0.$$
By Fatou's lemma, we have for $n\geq N_0$, 
\begin{align}\label{f4}
    \int V|f(u_n-w)|^p dx\leq\liminf\limits_{ m\to \infty}\int V|f(u_n-u_m)|^p dx<\epsilon
\end{align}
Using property 3 in lemma \ref{LL}, we can conclude $u_n\to w$ in $X$. Hence, $(X, ||\cdot||)$ is a Banach space. 
\end{proof}
\begin{rmk}\label{eqv}
    Under the condition $(v_1)$, $X=W^{1,p}(\R^N)$ and the $||\cdot||$ is equivalent to the usual norm $||\cdot||_{1,p,\R^N}$.
\end{rmk}
\begin{proof}
 Let $u\in W^{1,p}(\R^N)$. By property (ii) in lemma \ref{P}, we have 
 $$\int V(x)|f(u)|^p dx\leq V(\infty) \int |u|^p \;dx<\infty.$$
Hence, $X=W^{1,p}(\R^N)$. We claim that the identity map $Id: W^{1,p}(\R^N)\to X$ is a bounded linear map. 
Using property (ii) in lemma \ref{P} and $(v_1)$, we have 
\begin{align}
    \inf_{\xi>0}\frac{1}{\xi}[1+\int V(x)|f(\xi u)|^p\; dx]\leq \inf_{\xi>0}[\frac{1}{\xi}+(\xi)^{p-1}V(\infty)\int |u|^p\; dx]
\end{align}
Now, consider the function $$g(\xi)=\frac{1}{\xi}+L\xi^{p-1} \text{ for } \xi>0$$ where $L=V(\infty)\int |u|^p\; dx$. One can directly find the global minimum of $g$, which is equal to $[(p-1)^\frac{1}{p}+(p-1)^\frac{1-p}{p}]\;L^\frac{1}{p}.$
Thus, there exists a constant $C=[(p-1)^\frac{1}{p}+(p-1)^\frac{1-p}{p}]V(\infty)^\frac{1}{p}$ such that $$\inf_{\xi>0}\frac{1}{\xi}[1+\int V(x)|f(\xi u)|^p dx]\leq C||u||_{L^p(\R^N)}$$
which proves that the map $Id$ is bounded. The conclusion follows from the Inverse Mapping Theorem.
\end{proof}
The following compactness lemma is very useful whose proof is similar to that of lemma 2.2 \cite{YW}.
\begin{lem}\label{CP1}
    If $\{v_n\}$ is a bounded sequence in $X$ such that $$\sup_{x\in \R^N}\int_{B_1(x)}|f(v_n)|^p dx\rightarrow 0 \text{ as } n\rightarrow\infty.$$
    Then $f(v_n)\rightarrow 0$ in $L^s(\R^N)$ for every $s\in (p, 2\alpha p^*).$
\end{lem}
\begin{theorem}(Lions \cite[\text{Theorem II.1}]{PLL} )\label{ST}
    The embedding $W^{1,p}_r(\R^N)\hookrightarrow L^q(\R^N)$ is compact for $p<q<p^*$.
\end{theorem}
We will use the following slightly modified version of Jeanjean \cite[\text{Theorem 1.1}]{JJ}. The last part of the theorem follows from \cite[\text{Lemma 2.3}]{JJ}.
\begin{theorem}\label{MP}
      Let $(X,||\cdot||)$ be a Banach space and $J\subset\R^+$ be an interval. Consider the family of $C^1$ functional 
   \begin{align*}
        I_\delta(u)=Au-\delta Bu, \text{ }\delta\in J
   \end{align*}
    where $B$ is a non-negative functional and either $Au\to \infty$ or $Bu\to \infty$ as $||u||\to\infty$. If $$C_\delta:=\inf_{\gamma\in\Gamma_\delta}\max_{t\in[0,1]}I_\delta(\gamma(t))>0$$
    where $\Gamma_\delta=\{ \gamma\in C([0, 1];X) :\gamma(0)=0, I_\delta(\gamma(1))<0 \}$. Then for almost every $\delta\in J$, there exists a sequence $\{x_n\}$ such that 
    \begin{itemize}
        \item $\{x_n\}$ is bounded in $X$.
        \item $I_\delta(x_n)$ converges to $C_\delta$.
        \item $I'_\delta(x_n)$ converges to 0 in $X^*$.
    \end{itemize}
Moreover, the map: $\delta\to C_\delta$ is continuous from the left.
\end{theorem}
Now, we are ready to reformulate our problem. We define another functional $J:X\to \R$ by
\begin{align*}
     J(v)=\frac{1}{p}\int H(Dv)^p+ V(x)|f(v)|^p]dx-\frac{\lambda}{q+1}\int |f(v)|^{q+1} dx.
 \end{align*}
 which is a $C^1$ functional, whose derivative is given by
\begin{align}
    \langle J'(v),w\rangle=\int H(Dv)^{p-1}\nabla H(Dv).\nabla Dw dx +\int V|f(v)|^{p-2}f(v)f'(v)w dx - \lambda \int |f(v)|^{q-1}f(v)f'(v)w dx
\end{align}
If $v$ is a critical point of $J$ then $v$ satisfies the following equation in the weak sense
\begin{align}\label{maineq2}
    -\Delta_{H,p}v+V(x)|f(v)|^{p-2}f(v)f'(v)-\frac{\lambda}{q+1}|f(v)|^{q-1}f(v)f'(v)=0.
\end{align}
Thus $u=f(v)$ is a solution of (\ref{maineq}). Our aim is to find a critical point of $J$ in $X$. Firstly, we present Poho$\breve{z}$aev type identity corresponding to (\ref{maineq2}), and for that, we need the following lemma. 
\begin{lem}\label{Bdd}
Let $v\in W^{1,p}(\R^N)$ be a solution of (\ref{maineq2}) with $q\in \Pi$. Then $v\in L^\infty(\R^N)$.
\end{lem}
\begin{proof}
For each $m\in \N$ and $s>1$, we define $A_m=\{x\in\R^N| |u(x)|^{s-1}\leq m\}$, $B_m=\R^N\setminus A_m$ and let us consider two functions
\begin{align}
    v_m=\begin{cases}
    v|v|^{p(s-1)} \text{ if } x\in A_m\\
    m^pv         \text{ if } x\in B_m  
\end{cases}
\end{align}
and
\begin{align}
    w_m=\begin{cases}
    v|v|^{(s-1)} \text{ if } x\in A_m\\
    mv         \text{ if } x\in B_m  
\end{cases}
\end{align}
So we have that $v_m\in W^{1,p}(\R^N)$, $|v_m|\leq |v|^{ps-p+1}$, $||v|^{p-1}v|=|w_m|^p\leq m^p|v|^p$, $|w_m|\leq |v|^s$, 
\begin{align}
    \nabla v_m=\begin{cases}
    (ps-p+1)|v|^{p(s-1)}\nabla v \text{ if } x\in A_m\\
    m^p\nabla v         \text{ if } x\in B_m  
\end{cases}
\end{align}
and 
\begin{align}
    \nabla w_m=\begin{cases}
    s|v|^{s-1}\nabla v \text{ if } x\in A_m\\
    m\nabla v         \text{ if } x\in B_m  
\end{cases}
\end{align}
We also have 
\begin{align}\label{M0}
    \int H(Dw_m)^p dx&=\int_{A_m} H(s|v|^{s-1}Dv)^p dx + \int_{B_m} H(mDv)^p dx\nonumber\\
    &=\int_{A_m} s^p |v|^{p(s-1)}H(Dv)^p dx + m^p\int_{B_m} H(Dv)^p dx
\end{align}
and 
\begin{align}\label{M1}
    \int H(Dv)^{p-1}\nabla H(Dv).Dv_m dx&=(ps-p+1)\int_{A_m} |v|^{p(s-1)}H(Dv)^{p-1}\nabla H(Dv).Dv dx + m^p\int_{B_m} H(Dv)^p dx\nonumber\\
    &=(ps-p+1)\int_{A_m} |v|^{p(s-1)}H(Dv)^p dx + m^p\int_{B_m} H(Dv)^p dx
\end{align}
As a consequence of (\ref{M1}), we get
\begin{align}\label{M2}
    \int_{A_m} |v|^{p(s-1)}H(Dv)^p dx\leq\frac{1}{ps-p+1}\int H(Dv)^{p-1}\nabla H(Dv).Dv_m dx
\end{align}
 Using (\ref{M0}) and (\ref{M1}), we derive
 \begin{align}\label{M3}
   \int H(Dw_m)^p dx=\int H(Dv)^{p-1}\nabla H(Dv).Dv_m dx+(s^p-ps+p-1) \int_{A_m} |v|^{p(s-1)}H(Dv)^p dx
 \end{align}
 Consider $v_m$ as a test function and using the definition of weak solution, we obtain
 \begin{align}\label{M4}
     \int H(Dv)^{p-1}\nabla H(Dv).Dv_m=\int [\lambda|f(v)|^{q-1}-V(x)|f(v)|^{p-2}]f(v)f'(v)v_m dx
 \end{align}

Using (\ref{M2})-(\ref{M4}), we obtain
\begin{align}\label{M5}
   \int H(Dw_m)^p&dx+s^p\int V(x)|f(v)|^{p-2}f(v)f'(v)v_m dx=\int H(Dv)^{p-1}\nabla H(Dv).Dv_m dx+(s^p-ps+p-1)\nonumber\\ &\int_{A_m} |v|^{p(s-1)}H(Dv)^p dx+s^p\int V(x)|f(v)|^{p-2}f(v)f'(v)v_m dx\nonumber\\
   &\leq [\frac{s^p-ps+p-1}{ps-p+1}+1]\int H(Dv)^{p-1}\nabla H(Dv).Dv_m dx+s^p\int V(x)|f(v)|^{p-2}f(v)f'(v)v_m dx\nonumber\\
   &\leq s^p[\int H(Dv)^{p-1}\nabla H(Dv).Dv_m dx+\int V(x)|f(v)|^{p-2}f(v)f'(v)v_m dx]\nonumber\\
   &=s^p\lambda\int |f(v)|^{q-1}f(v)f'(v)v_m dx
\end{align}
 since $f'(t)\leq 1$, (\ref{M5}) implies 
\begin{align*}
     \int H(Dw_m)^p dx\leq s^p\lambda\int |f(v)|^q|v_m| dx
\end{align*}
By using the facts that $|f(t)|\leq |t|$, $|f(t)|\leq \frac{1}{2^{2\alpha p}}|t|^\frac{1}{2\alpha}$ and $||v|^{p-1}v_m|=|v_m|^p$, we have 
\begin{align}\label{M6}
     \int H(Dw_m)^p dx\leq s^p\lambda\int |f(v)|^{p-1}|f(v)|^{q-p+1}|v_m| dx &\leq s^p\lambda\int |v|^{p-1}|v_m||v|^\frac{q-p+1}{2\alpha} dx\nonumber\\&\leq s^p\lambda\int |w_m|^p|v|^\frac{q-p+1}{2\alpha} dx
\end{align}

Thus, it follows from H\"{o}lder inequality and $|w_m|\leq |v|^s$, that
\begin{align}\label{M7}
    \int H(Dw_m)^p dx\leq s^p\lambda (\int |w_m|^{pr} dx)^\frac{1}{r}(\int|v|^\frac{(q-p+1)r'}{2\alpha})^\frac{1}{r'}\leq s^p\lambda (\int |v|^{spr} dx)^\frac{1}{r}(\int|v|^{p^*} dx)^\frac{1}{r'}
\end{align}
where we choose $r>0$ such that $$\frac{(q-p+1)r'}{2\alpha}=p^*$$
Now, by applying the Sobolev's inequality and the inequality (\ref{M7}), we deduce
\begin{align}
    (\int_{A_m} |w_m|^{p^*})^\frac{p}{p^*}\leq C\int H(Dw_m)^p dx\leq s^p\lambda (\int |v|^{spr} dx)^\frac{1}{r}(\int|v|^{p^*} dx)^\frac{q-p+1}{2\alpha p^*}
\end{align}
Since $|w_m|=|v|^s$ in $A_m$, by using the monotone convergence theorem, we obtain
\begin{align}
    (\int |v|^{sp^*})^\frac{1}{sp^*}\leq  \{C\lambda\}^\frac{1}{sp}s^\frac{1}{s} (\int |v|^{spr} dx)^\frac{1}{spr}(\int|v|^{p^*} dx)^\frac{q-p+1}{2sp\alpha p^*}
\end{align}
that is,
\begin{align}\label{M8}
    ||v||_{L^{sp^*}(\R^N)}\leq (C\lambda)^\frac{1}{sp}s^\frac{1}{s}||v||_{L^{spr}(\R^N)}||v||_{L^{p^*}(\R^N)}^\frac{\tilde{r}}{s}
\end{align}
where $\tilde{r}=\frac{q-p+1}{2\alpha p}$.
We choose $\sigma=\frac{p^*}{pr}$ and observe that $\sigma>1$ if and only if $q<2\alpha p^*-2\alpha p+p-1$.\\
By taking $s=\sigma$ in (\ref{M8}), we obtain
\begin{align}\label{M9}
     ||v||_{L^{\sigma p^*}(\R^N)}\leq (C\lambda)^\frac{1}{\sigma p}\sigma^\frac{1}{\sigma}||v||_{L^{p^*}(\R^N)}||v||_{L^{p^*}(\R^N)}^\frac{\tilde{r}}{\sigma}
\end{align}
putting $s=\sigma^2$ and using (\ref{M9}), we have 
\begin{align}\label{M10}
     ||v||_{L^{\sigma^2p^*}(\R^N)}\leq (C\lambda)^{\frac{1}{p}[\frac{1}{\sigma}+\frac{1}{\sigma^2}]}\sigma^{[\frac{1}{\sigma}+\frac{2}{\sigma^2}]}||v||_{L^{p^*}(\R^N)}^{\tilde{r}[\frac{1}{\sigma}+\frac{1}{\sigma^2}]}
\end{align}

putting $s=\sigma^k$ and continuing the above process, we obtain

\begin{align}\label{M11}
     ||v||_{L^{\sigma^kp^*}(\R^N)} \leq (C\lambda)^{\frac{1}{p} \sum_{i=1}^{k}\frac{1}{\sigma^i}} \sigma^{\sum_{i=1}^{k}\frac{i}{\sigma^i}} ||v||_{L^{p^*}(\R^N)}^{\tilde{r}\sum_{i=1}^{k}\frac{1}{\sigma^i}}
\end{align}
By taking $k\to \infty$, we obtain
\begin{align}
    ||v||_{\infty}<\infty
\end{align}
Thus, we proved that if $p-1<q<2\alpha p^*-2\alpha p+p-1$ then $u\in L^\infty(\R^N)$.

Again, we can easily derive the following inequality from (\ref{M5}) by using the facts that $f'(t)\leq 1$, $|f(t)|\leq \frac{1}{2^{2\alpha p}}|t|^\frac{1}{2\alpha}$ and $||v|^{p-1}v_m|=|v_m|^p$.
\begin{align}\label{M12}
     \int H(Dw_m)^p dx&\leq s^p\lambda 2^\frac{q}{2\alpha p}\int |w_m|^p|v|^{\frac{q}{2\alpha}-p+1} dx \nonumber\\
\end{align}
Similar argument as before we can prove  $u\in L^\infty(\R^N)$ if $2\alpha p-2\alpha<q<2\alpha p^*-2\alpha$.
Since $|f'(t)|\leq \frac{C}{|f(t)|^{2\alpha-1}}$, (\ref{M5}) implies
\begin{align}
     \int H(Dw_m)^p dx\leq Cs^p\lambda\int |f(v)|^{q-2\alpha+1}|v_m| dx
\end{align}
If $q$ satisfies one of the following conditions
\begin{enumerate}[label=(\alph*)]
    \item $p+2\alpha-2<q<2\alpha p^*-2\alpha p+2\alpha-2$
    \item $2\alpha p-1<q<2\alpha p^*-1$
    \item $\frac{p-1}{2\alpha}+2\alpha-1<q<p^*-p+\frac{p-1}{2\alpha}+2\alpha-1$
\end{enumerate}
then by similar argument one can prove that $u\in L^\infty(\R^N)$.
\end{proof}
\begin{lem}\label{PI}
    If $v$ is a critical point of the functional $J$ then $$P_V(v)=\frac{N-p}{p}\int H(Dv)^p dx+ \frac{N}{p}\int V(x)|f(v)|^p dx+\frac{1}{p}\int \langle\nabla V(x), x\rangle |f(v)|^p dx-\frac{\lambda N}{q+1}\int |f(v)|^{q+1} dx=0.$$
    We denote $P_V(v)$ as Poho$\breve{z}$aev identity.
\end{lem}
\begin{proof}
    The preceding lemma ensures us that $v\in L^\infty(\R^N)$ and using \cite[\text{ Proposition 4.3}]{CM}, we obtain $v\in C^1(\R^N)$. Hence, this result follows from \cite[\text{ Theorem 1.3}]{LS}. 
\end{proof}
\begin{cor}
If V is constant then the Poho$\check{z}$aev identity becomes 
\begin{align}\label{poho}
    P(v)=\frac{N-p}{p}\int H(Dv)^p dx+ \frac{N}{p}\int V(x)|f(v)|^p dx-\frac{\lambda N}{q+1}\int |f(v)|^{q+1} dx=0.
\end{align}
\end{cor}

\section{Auxiliary results}\label{Aux}
In this section, we prove a few auxiliary results that will be repeatedly used in the future. To begin with, we state two important lemmas whose proof can be found in (Bal et al. \cite{FC}) or in the references therein.
\begin{lem}(\cite[ \text{ lemma 2.1 }]{KB})\label{ET}
    Let $x\in \R^n\setminus\{0\}$ and $t\in\R\setminus\{0\}$ then
    \begin{enumerate}[label=(\alph*)]
        \item $x\cdot\nabla_\xi H(x)=H(x)$.
        \item $\nabla_\xi H(tx)=\mbox{sign}(t) \nabla_\xi H(x)$.
        \item $||\nabla_\xi H(x)||\leq C$, for some positive constant $C$.
        \item $H$ is strictly convex.
    \end{enumerate}
\end{lem}
\begin{lem}(\cite[ \text{ lemma 2.5 }]{KB})\label{ET1}
    Let $2\leq p<\infty$. Then for $x, y\in \R^N$, there exists a positive constant $C$ such that
    \begin{align}\label{IN1}
        \langle H(x)^{p-1}\nabla_\eta H(x)-H(y)^{p-1}\nabla_\eta H(y), x-y\rangle\geq C\;H(x-y)^p.
    \end{align}
\end{lem}
Next, we prove the following lemma, which will assist us in drawing conclusions regarding the pointwise convergence of the gradient of a Palais-Smale sequence of $J$ in $X$.
\begin{lem}\label{PBS}
    Let $p\geq 2$ and define $$T(t)=\begin{cases}
        t\; \text{ if }\; |t|\leq 1\\
        \frac{t}{|t|} \;\mbox{ otherwise }
    \end{cases}$$
    and assume that $[H1]-[H5]$ hold. Let $\{v_n\}$ be a sequence in $D^{1,p}(\R^N)$ such that $v_n\rightharpoonup v$ in $D^{1,p}(\R^N)$ and for every $\phi\in C^\infty_c(\R^N)$, $$ \int \phi (H(Dv_n)^{p-1}\nabla H(Dv_n)-H(Dv)^{p-1}\nabla H(Dv)).\nabla T(v_n-v) dx\rightarrow 0.$$
    Then up to a subsequence, the following conclusions hold:
    \begin{enumerate}[label=(\alph*)]
        \item $Dv_n\rightarrow Dv$ a.e. in $\R^N$.
        \item $\lim_{n\rightarrow\infty}[||H(Dv_n)||_{L^p}^p-||H(Dv_n-Dv)||_{L^p}^p]=||H(Dv)||_{L^p}^p$.
        \item  $H(Dv_n)^{p-1}\nabla H(Dv_n)-H(Dv_n-Dv)^{p-1}\nabla H(Dv_n-Dv)\rightarrow H(Dv)^{p-1}\nabla H(Dv)$\;\mbox{in}\; $L^{\frac{p}{p-1}}(\R^N)$.
    \end{enumerate}
\end{lem}

\begin{proof}
\begin{enumerate}[label=(\alph*)]
    \item  Let us define $w_n=(H(Dv_n)^{p-1}\nabla H(Dv_n)-H(Dv)^{p-1}\nabla H(Dv)).(\nabla v_n-\nabla v)\geq 0$. Let $\phi\in C_c^\infty(\R^N)$ be a nonnegative function and $\Omega=\mbox{supp}(\phi)$. Without loss of generality, we can assume $v_n\rightarrow v$ in $L^p(\Omega)$ and $v_n\rightarrow v$ almost everywhere in $\R^N$. Using the given condition and H\"{o}lder inequality, we have for every $s\in(0, 1)$
\begin{align}\label{sa}
    0\leq \int (\phi w_n)^s dx&\leq \int_{K_n} (\phi w_n)^s dx +\int_{L_n} (\phi w_n)^s dx\nonumber\\
    &\leq |K_n|^{1-s}(\int_{K_n}\phi w_n dx)^s+ |L_n|^{1-s}(\int_{L_n}\phi w_n dx)^s\nonumber\\
    &\leq |\Omega|^{1-s} o(1)+o(1)
\end{align}
where $K_n=\{x\in \Omega: |u_n(x)-u(x)|\leq 1\}$ and $L_n=\{x\in \text{supp}(\phi): |u_n(x)-u(x)|>1\}$.\\ From (\ref{sa}) one has $w_n\rightarrow 0$ a.e in $\R^N$. Hence remark \ref{equiv} and lemma \ref{ET1} ensure $Dv_n\rightarrow Dv$ almost everywhere in $\R^N$.
\item It follows from Brezis-Leib lemma \cite{BL}.
\item Define $G:\R^N\to\R^N$ by $$G(x)=H(x)^{p-1}\nabla H(x)$$ and let  $$G_i(x)=H(x)^{p-2}HH_{x_i}.$$ By using lemma \ref{ET} and $(H4)$, we have
\begin{align}
    |G(x+h)-G(x)|=|\int_{0}^{1}\frac{d}{dt}[G(x+th)] dt|&\leq C\int_{0}^{1} H(x+th)^{p-2}|h| dt\nonumber\\
    &\leq C\int_{0}^{1} [H(x)^{p-2}+t^{p-2}H(h)^{p-2}]|H(h)| dt\nonumber\\
    &\leq C[H(x)^{p-2}H(h)+H(h)^{p-1}]\leq \epsilon H(x)^{p-1}+C_\epsilon H(h)^{p-1}\nonumber
\end{align}
where $\epsilon>0$ be any real number and $C_\epsilon>$ is a constant. Finally, we get \begin{align}\label{Y1}
    |G(x+h)-G(x)|< \epsilon H(x)^{p-1}+C_\epsilon H(h)^{p-1}
\end{align}
We define $$\Psi_{\epsilon,n}:=[|G(Dv_n)-G(Dv_n-Dv)-G(Dv)|-\epsilon H(Dv_n)^{p-1}]^+.$$
Clearly, $\Psi_{\epsilon,n}\rightarrow 0$ as $n\rightarrow\infty$ almost everywhere in $\R^N$. Using (\ref{Y1}) and Remark \ref{equiv}, we have $$\Psi_{\epsilon,n}\leq G(Dv)+C_\epsilon H(Dv)^{p-1}\leq C H(Dv)^{p-1}.$$
By the Dominated Convergence Theorem, we get
\begin{align}\label{Y2}
    \lim\int |\psi_{\epsilon,n}|^\frac{p}{p-1}dx=0.
\end{align}
From the definition of $\phi_{\epsilon,n}$ 
\begin{align*}
    |G(Dv_n)-G(Dv_n-Dv)-G(Dv)|\leq \Psi_{\epsilon,n}+\epsilon H(Dv_n)^{p-1}
\end{align*}
Since $\{v_n\}$ is bounded in $D^{1,p}(\R^N)$, using (\ref{Y2}) we conclude
$$\limsup_{n\to \infty}\int |G(Dv_n)-G(Dv_n-Dv)-G(Dv)|^\frac{p}{p-1} dx\leq \epsilon M$$
for some $M>0$. As $\epsilon>0$ is arbitrary so (c) follows.
\end{enumerate}
\end{proof}
\begin{rmk}
\begin{enumerate}[label=(\roman*)]
    \item The condition (H5) is not required to conclude $(a)$ and $(b).$
    \item  The above result is true for $1<p<\infty$ if $H=H_2$.
\end{enumerate}   
\end{rmk}
\begin{lem}\label{BL}
 Suppose that $p=2$ and $2\alpha\leq p$; or $p\geq 2\alpha+1$ hold 
 and $G$ is function from $\R$ to $\R$ such that $G(t)=|f(t)|^{p-2}f(t)f'(t).$ If $v_n\rightharpoonup v$ in $X$ then
       $$\lim_{n\to \infty}\int |G(v_n)-G(v_n-v)-G(v)|^\frac{p}{p-1} dx=0.$$
\end{lem}
\begin{proof}
    The proof is the same as that of the previous lemma. we omit it here.
\end{proof}
\begin{lem}\label{AL}
     Let us consider the function $h:(0, \infty)\to \R$,\; $h(t)=C_1t^{N-p}+t^N(C_2-\lambda C_3)$, where $C_1, C_2,\; \mbox{and}\; C_3$ are positive constants and $N>p$. Then for large enough $\lambda>0$, $h$ has a unique critical point, which corresponds to its maximum.
\end{lem}
\begin{proof}
The proof is very simple and is omitted here.
\end{proof}
\begin{rmk}
    Notice that if $h$ has a critical point $t_0>0$ then $C_2-\lambda C_3<0$ and $h(t_0)=\max_{t>0}h(t).$ 
\end{rmk}

We introduce the Poho$\check{z}$aev manifold 
$$M=\{v\in X_r\setminus\{0\}: P(v)=0\}$$
where $P_V$ is defined in (\ref{poho}).

\begin{lem}\label{NA}
    Let $N\geq 2$, $(\alpha, p)\in D_N$ and $p-1<q<2\alpha p^*-1$. Then 
\begin{enumerate}[label=(\roman*)]
    \item For any $v\in X_r\setminus\{0\}$, there exists unique $t_0=t_0(v)>0$ such that $v_{t_0}=v(\frac{.}{t_0})\in M$. Moreover, $J(v_{t_0})=\max_{t>0}J(v_t).$ 
    \item $0\notin \partial M$ and $\inf_{v\in M} J(v)>0.$
    \item For any $v\in M$, $P'(v)\neq 0.$
    \item $M$ is a natural constraint of $J$.
\end{enumerate}
\end{lem}

\begin{proof}
\begin{enumerate}[label=\roman*)]
    \item  Let $v\in X_r\setminus\{0\}$. For $t>0$, we define $v_t(x)=v(\frac{x}{t})$. Now, consider the function $$\phi:(0, \infty)\to \R\;\mbox{by}\; \phi(t)=J(v_t).$$ After simplification, we have $$\phi(t)=\frac{t^{N-p}}{p}\int H(Dv)^p dx+ \frac{t^N}{p}\int V(x)|f(v)|^p dx-\frac{\lambda t^N}{q+1}\int |f(v)|^{q+1} dx.$$
By lemma\ref{AL}, for large $\lambda>0$, there exists $t_0>0$ such that $\phi'(t_0)=0$ and $\phi(t_0)=\max_{t>0}\phi(t).$ Also we notice that $P(v_{t_0})=t_0\phi'(t_0)=0$. Hence, $v_{t_0}\in M$ and $J(v_{t_0})=\max_{t>0}J(v_t).$
\item  If $v\in M$ then
\begin{align*}
    \frac{N-p}{p}\int [H(Dv)^p+V|f(v)|^p]dx-\frac{\lambda N}{q+1}\int |f(v)|^{q+1} dx\leq P(v)=0
\end{align*}
which ensures 
\begin{align}\label{Ineq}
    \frac{N-p}{p}\int [H(Dv)^p+V|f(v)|^p]dx\leq \frac{\lambda N}{q+1}\int |f(v)|^{q+1} dx
\end{align}
Now, let $\int [H(Dv)^p+V|f(v)|^p]dx=\beta^p$ and $\gamma>0$ (to be determined later).  By using H\"{o}lder inequality and Sobolev inequality, we obtain 
\begin{align}\label{ineq1}
    \int |f(v)|^{q+1} dx\leq \int |f(v)|^{q+1} dx&\leq [\int |f(v)|^p dx]^\frac{\gamma(q+1)}{p} [\int |f^{2\alpha}(v)|^{p^*}]^{1-\frac{\gamma(q+1)}{p}} \nonumber\\
    &\leq C [\int |f(v)|^p dx]^\frac{\gamma(q+1)}{p} [\int |Dv|^p dx]^{\frac{p^*}{p}(1-\frac{\gamma(q+1)}{p})}\nonumber\\
    &\leq C [\int |f(v)|^p dx]^\frac{\gamma(q+1)}{p} [\int H(Dv)^p dx]^{\frac{p^*}{p}(1-\frac{\gamma(q+1)}{p})}\\
    &\leq C \beta^m\nonumber
\end{align}
where $\gamma=\frac{p(2\alpha p^*-q-1)}{(q+1)(2\alpha p^*-p)}$ and $m=\frac{2\alpha p p^*-pq-p+p^*(q+1-p)}{2\alpha p^*-p}>p.$ As $m>p$ so by using the above inequality and (\ref{Ineq}), we get $\beta\geq C$, for some positive constant $C$. Hence, $0\notin \partial M$. Notice that if $v\in M$ then $NJ(v)-P(v)=\int H(Dv)^p dx>0$. So, $J(v)>0$, for all $v\in M$. We shall prove that $\inf_{v\in M} J(v)>0.$ If not then there exists a sequence $\{v_n\}$ in $M$ such that $J(v_n)\to 0.$ We can prove that the sequence $\{v_n\}$ is bounded (see the proof of Theorem \ref{TC}). Using (\ref{ineq1}) and $\lim_{n\to\infty}\int H(Dv_n)^p dx=\lim_{n\to\infty} NJ(v_n)=0$, we get 
\begin{align}\label{f1}
    \int |f(v_n)|^{q+1} dx\to 0 \text{ as } n\to \infty
\end{align}
Now (\ref{f1}) and $\lim_{n\to\infty}J(v_n)=0$ implies $0\in \partial M$, which is a contradiction.

\item If possible let $P'(v)=0$, for some $v\in M$. As $v\in M$ and $r=J(v)>0$ so
\begin{align}\label{P1}
    \frac{N-p}{p}\int H(Dv)^p dx+ \frac{N}{p}\int V|f(v)|^p dx-\frac{\lambda N}{q+1}\int |f(v)|^{q+1} dx=0
    \end{align}
    and
    \begin{align}\label{P2}
    \frac{1}{p}\int [H(Dv)^p+ V|f(v)|^p]dx-\frac{\lambda}{q+1}\int |f(v)|^{q+1} dx=r
\end{align}
As $P'(v)=0$ so $v$ satisfies the following equation in weak sense
$$-(N-p)\Delta_{H,p}v+NV|f(v)|^{p-2}f(v)f'(v)-\lambda N|f(v)|^{q-1}f(v)f'(v)=0.$$
Hence $v$ satisfies the corresponding Poho$\Check{z}$aev identity:
\begin{align}\label{P3}
     \frac{(N-p)^2}{p}\int H(Dv)^p dx+ \frac{N^2}{p}\int V|f(v)|^p dx-\frac{\lambda N^2}{q+1}\int |f(v)|^{q+1} dx=0
\end{align}
For simplicity, we assume $a=\int H(Dv)^p dx$, $b=\int V|f(v)|^p dx$, $c=\int |f(v)|^{q+1} dx$. By using (\ref{P1}), (\ref{P3}) and $p<N$, we have $a=0$. From (\ref{P1}), we obtain $$\frac{b}{p}=\frac{\lambda c}{q+1}.$$
By using the above result, (\ref{P2}) gives us $r=0$, which is a contradiction as $r>0$.\\
\item Let $v\in M$ such that $J(v)=\inf_{u\in M}J(u)=r>0$ (say). Our Claim is $J'(v)=0$ in $X^*$. By Lagrange's multiplier, there exists a $\tau\in \R$ such that $J'(v)=\tau P'(v)$.
So, $v$ satisfies the following equation in the weak sense
\begin{align}
    -(1-\tau(N-p))\Delta_{H,p}v+(1-\tau N)V|f(v)|^{p-2}f(v)f'(v)+\lambda(\tau N-1)|f(v)|^{q-1}f(v)f'(v)=0.
\end{align}
Hence, $v$ satisfies the corresponding Poho$\Check{z}$aev identity. Using the same notation as before we have the following equations 
\begin{align}\label{P4}
    \frac{(N-p)(1-\tau(N-p))}{p}a+\frac{(1-\tau N)N}{p}b-\frac{N\lambda(1-\tau N)}{q+1}c=0.
\end{align}
\begin{align}\label{P5}
     \frac{a}{p}+\frac{b}{p}-\frac{\lambda c}{q+1}=r.
\end{align}
and
\begin{align}\label{P6}
    \frac{N-p}{p}a+\frac{N}{p}b-\frac{\lambda N}{q+1}c=0.
\end{align}
If $\tau\neq 0$ then by using (\ref{P4}), (\ref{P5}) and (\ref{P6}), we have $r=0$; which contradicts the fact that $r>0$. Hence, $\tau=0$. Consequently, $J'(v)=0$ in $X'$.
\end{enumerate}
\end{proof}
Let us consider a collection of auxiliary functional on $X$,\;
$\{J_\delta\}_{\delta \in I}$ of the form
\begin{align}\label{AF1}
    J_\delta(v)=\frac{1}{p}\int H(Dv)^p+ V(x)|f(v)|^p]dx-\frac{\lambda\delta}{q+1}\int |f(v)|^{q+1} dx
\end{align}
and we define
\begin{align}\label{AF2}
    J_{\infty,\delta}(v)= \frac{1}{p}\int H(Dv)^p+ V(\infty)|f(v)|^p]dx-\frac{\lambda \delta}{q+1}\int |f(v)|^{q+1} dx.
\end{align}

\begin{lem}\label{NL}
  Assume the potential $V$ satisfies $(v_1)$. Then the set  $$\Gamma_\delta=\{\gamma\in C([0, 1]; X): \gamma(0)=0,\; J_\delta(\gamma(1))<0\}\neq \{0\},\;\mbox{for any}\; \delta\in I.$$
\end{lem}
\begin{proof}
    For every $v\in X$,
\begin{align}
    J_\delta(v)\leq J_{\infty, \frac{1}{2}}(v)
\end{align}
Now, let $v\in X\setminus \{0\}$,
    $$J_{\infty,\frac{1}{2}}(v_t)= J_{\infty,\frac{1}{2}}(v(\frac{x}{t}))= \frac{t^{N-p}}{p}\int H(Dv)^p dx+\frac{t^N}{p}\int V(\infty)|f(v)|^p]dx-\frac{\lambda t^N}{2(q+1)}\int |f(v)|^{q+1} dx$$
As $\lambda>0$ is large enough so $J_{\infty,\frac{1}{2}}(v_t)\to -\infty$ as $t\to \infty$. Hence, there exists $t_0>0$ such that
$J_{\infty,\frac{1}{2}}(v_{t_0})<0$. Consequently, $J_\delta(v_{t_0})<0$, for all $\delta\in I$.
Define $\gamma:[0,1 ]\to X$ as
$$\gamma(t)=\begin{cases}
    0, \text{ if } t=0\\
    (v_{t_0})_t, \text{ if } 0<t\leq 1
\end{cases}$$
It is easy to prove that the $\gamma$ is continuous. Hence, $\gamma$ is a desired path.
\end{proof}

\begin{lem}
    The above collection of functional $J_\delta$ satisfies all the hypotheses of Theorem \ref{MP}.
\end{lem}
\begin{proof}
    Here, $Au=\frac{1}{p}\int [H(Du)^p+V(x)|f(Du)|^p]\;dx$ and $Bu=\frac{\lambda \delta}{q+1}\int |f(u)|^{q+1}\;dx.$ Clearly, $B$ is nonnegative and $Au\to \infty$ as $||u||\to \infty$.\\ Claim: 
\begin{align}\label{Cdef}
    C_\delta=\inf_{\gamma\in\Gamma} \max_{t\in[0, 1]} J_\delta(\gamma(t))>0.
\end{align}
Let $u\in S(\beta):=\{u\in X: \int [H(Du)^p+V(x)|f(Du)|^p]\; dx=\beta^p$. By using (\ref{ineq1}), we have 
    \begin{align}
        J_\delta(u)&=\frac{1}{p}\int [H(Du)^p+V(x)|f(Du)|^p]\; dx-\frac{\lambda \delta}{q+1}\int |f(u)|^{q+1} \;dx\\
        &\geq \frac{1}{p}\beta^p -C\beta^m
    \end{align}
    where $m>p$. If $\beta>0$ is small enough then there exists $r>0$ such that $J_\delta(u)\geq r$ and hence $C_\delta\geq r>0$.   
\end{proof}
Define the Poho$\check{z}$aev Manifold
$$M_{\infty,\delta}=\{v\in X_r\setminus\{0\} : P_{\infty, \delta}(v)= \frac{N-p}{p}\int H(Dv)^p\; dx+ \frac{N}{p}\int V(\infty)|f(v)|^p\; dx-\frac{\lambda \delta N}{q+1}\int |f(v)|^{q+1}\; dx=0\}.$$
We have the following lemma:
\begin{lem}\label{LM}
    If $N\geq 2$, $(\alpha, p)\in D_N$, $p-1<q<2\alpha p^*-1$ and $\delta\in I$. Then for large $\lambda>0$, there exists $v_{\infty,\delta}\in M_{\infty,\delta}$ such that 
    $$J_{\infty, \delta}(v_{\infty,\delta})=m_{\infty,\delta}:=\{J_{\infty, \delta}(v) : v\neq 0, J'_{\infty, \delta}(v)=0\}.$$
    Moreover, $$J'_{\infty, \delta}(v_{\infty,\delta})=0.$$
\end{lem}

\begin{proof}
    This lemma is a simple consequence of Theorem \ref{TC} so we omit the proof.
\end{proof}

Now we are going to prove the following lemma
\begin{lem}\label{SL}
If the potential $V$ satisfies $(v_1)$ and $(v_2)$, $N\geq 2$, $(\alpha, p)\in D_N$ and $p-1<q<2\alpha p^*-1$. Then for every $\delta\in I$, $$C_\delta<m_{\infty, \delta}$$
\end{lem}
\begin{proof}
     It is easy to see that $$J_{\infty, \delta}(v_{\infty, \delta})=\max_{t>0}J_{\infty, \delta}(v_{\infty, \delta}(\frac{.}{t}))$$
     Let $\gamma$ be the curve defined in the lemma \ref{NL} for $v=v_{\infty,\delta}$. By using $(v_1)$, we have
            $$C_\delta\leq \max_{t\in[0,1]} J_\delta(\gamma(t))\leq \max_{t\in[0,1]} J_{\infty,\delta}(\gamma(t))\leq  J_{\infty,\delta}(v_{\infty,\delta})=m_{\infty,\delta}.$$

If possible let $C_\delta=m_{\infty,\delta}$. Then $\max_{t\in[0,1]} J_\delta(\gamma(t))= J_{\infty,\delta}(v_{\infty,\delta})$. As $m_{\infty,\delta}>0$ and $J_\delta o \gamma$ is a continuous map, so there exists $t^*\in (0,1)$ such that 
 $J_\delta(\gamma(t^*))=J_{\infty,\delta}(v_{\infty,\delta})=m_{\infty,\delta}$. Moreover, since $t=1$ is the unique maxima of $J_{\infty,\delta} o \gamma$ so $J_\delta(\gamma(t^*))>J_{\infty,\delta}(\gamma(t^*))$, which is not possible because of $(v_1)$. Hence, $C_\delta<m_{\infty,\delta}.$ 
 \end{proof}
 Now, we present the most important lemma to prove our main result.
\begin{lem}(Global compactness lemma)\label{GCL}
Suppose that $(v_1)$ and $(v_2)$ hold, $N\geq 3$, $(\alpha,p)\in D^N$, and $2\alpha p-1\leq q<2\alpha p^*-1$. For every $\delta\in I$, let $\{u_n\}$ be a bounded $(PS)_{c_\delta}$ sequence for $J_\delta$. Then there exist a subsequence of $\{u_n\}$, still denote by $\{u_n\}$ and $u_0\in X$, an integer $k\in \N\cup\{0\}$, $w_i\in X$, sequence $\{x^i_n\}\subset\R^N$ for $1\leq i\leq k$ such that
\begin{enumerate}[label=(\roman*)]
    \item $u_n\rightharpoonup u_0$ in $X$ with $J_\delta(u_0)\geq 0$ and $J'_\delta(u_0)=0$
    \item $|x^i_n|\to\infty$, $|x^i_n-x^j_n|\to\infty$ as $n\to \infty$ if $i\neq j$.
    \item $w_i\nequiv 0$ and $J_{\infty,\delta}'(w_i)=0$, for $1\leq i\leq k$.
    \item $||u_n-u_0-\sum_{i=1}^{k}w_i(.-x^i_n)||\to 0$ as $n\to \infty$.
    \item $J_\delta(u_n)\to J_\delta(u_0)+\sum_{i=1}^{k} J_{\infty,\delta}(w_i)$.
\end{enumerate}
\end{lem}
\begin{proof}
The proof consists of several steps:
\begin{enumerate}[label= Step \arabic*.]
\item Since $\{u_n\}$ is bounded so without loss of generality we can assume that 
\begin{enumerate}
    \item $u_n\rightharpoonup u_0\; \mbox{ in }\; X.$
    \item $f(u_n)\rightarrow f(u_0) \;\mbox{ in }\; L^s_{loc}(\R^N)\; \mbox{ for all}\; p\leq s\leq 2\alpha p^*.$
    \item $u_n\rightharpoonup u_0 \;\mbox{ a.e. in }\; \R^N.$
\end{enumerate}
We show that for every $\phi\in C_c^\infty(\R^N)$,$$\langle J'_\delta(u_0), \phi\rangle=\lim_{n\to\infty} \langle J'_\delta(u_n), \phi\rangle.$$
That is, we only have to show the following identities:
\begin{enumerate}[label=(\alph*)]
    \item $\lim_{n\to\infty}\int H(Du_n)^{p-1}\nabla H(Du_n).\nabla \phi\; dx=\int H(Du_0)^{p-1}\nabla H(Du_0).\nabla \phi\; dx$
    \item $\lim_{n\to\infty}\int V(x)|f(u_n)|^{p-2}f(u_n)f'(u_n)\phi\; dx=\int V(x)|f(u_0)|^{p-2}f(u_0)f'(u_0)\phi \;dx$
    \item $\lim_{n\to\infty}\int |f(u_n)|^{q-1}f(u_n)f'(u_n)\phi\; dx=\int V(x)|f(u_0)|^{q-1}f(u_0)f'(u_0)\phi\; dx$
\end{enumerate}
   Let $K=\mbox{supp}(\phi)$. For every $s\in[p, 2\alpha p^*)$, there exists $h_s\in L^s(K)$ such that up to a subsequence $|f(u_n)|,\; |f(u_0)|\leq h$ and $u_n\rightarrow u_0$ a.e. in $\R^N$. The equalities (b) and (c) are two consequences of the Dominated Convergence Theorem. As for the first part by Egorov's theorem for every $\epsilon>0$, there exists a measurable set $E\subset K$ such that $|E|<\epsilon$ and $u_n$ converges to $u$ uniformly on $E^c\cap K$. So for large n, $|u_n(x)-u(x)|\leq 1$. Using the fact that $u_n\rightharpoonup u_0$, we have 
\begin{align*}
|\int \phi H(Du_0)^{p-1}\nabla H(Du_0).\nabla T(u_n-u_0)\; dx|&\leq |\int_{E} \phi H(Du_0)^{p-1}\nabla H(Du_0).\nabla T(u_n-u_0)\; dx|\\ 
&+ |\int_{E^c\cap K} \phi H(Du_0)^{p-1}\nabla H(Du_0).\nabla T(u_n-u_0)\; dx|\\
& \leq \int_{E} |\phi H(Du_0)^{p-1}\nabla H(Du_0).\nabla T(u_n-u_0)|\; dx \\
& + |\int_{E^c\cap K} \phi H(Du_0)^{p-1}\nabla H(Du_0).\nabla(u_n-u_0) \;dx|\\
& \leq M\epsilon^\frac{1}{p}+o(1)
\end{align*}
Hence,
\begin{align}\label{I1}
    \int \phi H(Du_0)^{p-1}\nabla H(Du_0).\nabla T(u_n-u_0)\; dx\rightarrow 0.
\end{align}
Since $\{u_n\}$ is a bounded Palais-Smale sequence so 
    $$\langle J'_\delta(u_n), \phi.T(u_n-u_0)\rangle=o(1)$$
which implies 
\begin{align}\label{I2}
\int H(Du_n)^{p-1}\nabla H(Du_n).\nabla (\phi.T(u_n-u_0))\; dx&= -\int V|f(u_n)|^{p-2}f(u_n)f'(u_n)\phi T(u_n-u_0)\; dx\nonumber\\
&+ \lambda\delta\int |f(u_n)|^{q-1}f(u_n)f'(u_n)\phi T(u_n-u_0)\; dx + o(1)
\end{align}
Using (\ref{I2}), we have
\begin{align}\label{I3}
    |\int \phi H(Du_n)^{p-1}\nabla H(Du_n).\nabla T(u_n-u_0) \;dx|&\leq |\int H(Du_n)^{p-1}\nabla H(Du_n).\nabla (\phi.T(u_n-u_0)) \;dx|\nonumber\\
    &+ \int H(Du_n)^{p-1}(\nabla H(Du_n).\nabla\phi) T(u_n-u_0)\;dx|\nonumber\\
    &\leq \int V|f(u_n)|^{p-1}|\phi T(u_n-u_0)|\; dx\nonumber\\
    &+ \lambda\delta\int |f(u_n)|^q|f'(u_n)||\phi T(u_n-u_0)| \;dx\nonumber \\
    &+ |\int H(Du_n)^{p-1}(\nabla H(Du_n).\nabla\phi) T(u_n-u_0)\;dx|+ o(1)\nonumber\\
    &=o(1)
\end{align}
From (\ref{I1})-(\ref{I3}), one has
$$ \int \phi (H(Du_n)^{p-1}\nabla H(Du_n)-H(Du_0)^{p-1}\nabla H(Du_0)).\nabla T(u_n-u_0) dx\rightarrow 0.$$
By lemma \ref{PBS}, we can conclude $Du_n\rightarrow Du_0$ almost everywhere in $\R^N$. Moreover, 
\begin{enumerate}[label=\roman*)]
    \item $H(Du_n)^{p-1}\nabla H(Du_n)$ is bounded in $L^\frac{p}{p-1}(\R^N)$
    \item $H(Du_n)^{p-1}\nabla H(Du_n)\rightarrow H(Du_0)^{p-1}\nabla H(Du_0)$ a.e. in $\R^N$.
\end{enumerate}
Hence, $H(Du_n)^{p-1}\nabla H(Du_n)\rightharpoonup H(Du_0)^{p-1}\nabla H(Du_0)$ in $L^\frac{p}{p-1}(\R^N)$. Consequently, (a) follows.
\item
 Since $J_\delta'(u_0)=0$ so $u_0$ satisfies the following Poho$\check{z}$aev identity 
 \begin{align}\label{I4}
     \frac{N-p}{p}\int H(Du_0)^p + \frac{1}{p}\int [V(x)+\langle\nabla V(x),x\rangle]|f(u_0)|^p dx - \frac{\lambda\delta N}{q+1}\int |f(u_0)|^{q+1} dx=0.
 \end{align}
From (\ref{I4}) and $\langle J_\delta'(u_0),\frac{f(u_0)}{f'(u_0)}\rangle=0$ we deduce
\begin{align}\label{I5}
    \frac{N-p}{p}A + \frac{N}{p}\beta_2+\frac{1}{p}B-\frac{\lambda\delta N}{q+1}\beta_3=0.
\end{align}
and
\begin{align}\label{I6}
J_\delta(u_0)= (2\alpha-1)(\beta_1+\beta_2)+\frac{\lambda\delta\beta_3(q+1-2\alpha p)}{2\alpha p(q+1)}.
\end{align}
where $\beta_1=\int \frac{H(Du_0)^p}{1+(2\alpha)^{p-1}|f(u_0)|^{p(2\alpha-1)}} dx$, $\beta_2=\int V(x)|f(u_0)|^p dx$, $\beta_3=\int |f(u_0)|^{q+1} dx$, $A=\int H(Du_0)^p dx$, and $B=\int \langle\nabla V(x),x\rangle|f(u_0)|^p dx$.
Using (\ref{I5}), (\ref{I6}) and $(v_2)$, we get
\begin{align}
    NJ_\delta(u_0)=(2\alpha-1)(\beta_1+\beta_2)+\frac{(N-p)(q+1-2\alpha p)}{2\alpha p^2}A+\frac{q+1-2\alpha p}{2\alpha p^2}(N\beta_2+B)\geq 0.
\end{align}
\item
Since $\{u_n\}$ is a Palais-Smale sequence and $u_0$ is a critical point of $J_\delta$, we have  
\begin{align}\label{f5}
    \langle J'_\delta(u_n),\frac{f(u_n)}{f'(u_n)}\rangle&=\int H(Du_n)^p(1+G(u_n)) dx + \int V|f(u_n)|^p dx-\lambda\delta\int |f(u_n)|^{q+1} dx=o(1)
\end{align}
and 
\begin{align}\label{f6}
    \langle J'_\delta(u_0),\frac{f(u_0)}{f'(u_0)}\rangle&=\int H(Du_0)^p(1+G(u_0)) dx + \int V|f(u_0)|^p dx-\lambda\delta\int |f(u_0)|^{q+1} dx=0
\end{align}
where $G(t)=(2\alpha-1)(2\alpha)^{p-1}|f(t)|^{p(2\alpha-1)}[1+(2\alpha)^{p-1}|f(t)|^{p(2\alpha-1)}]^{-1}$.\\
Let $$\rho=\limsup_{n\to\infty} \sup_{ y\in\R^N}\int_{ B_1(y)} |f(u^1_n)|^p dx$$ where $u^1_n=u_n-u_0\rightharpoonup 0$ in $X$.
\begin{enumerate}
    \item \textbf{Vanishing Case:} If $\rho=0$ then by lemma \ref{CP1}, we have 
\begin{align*}
    f(u^1_n)\to 0\text{ in } L^{q+1}(\R^N).
\end{align*}
which implies 
\begin{align}\label{f7}
    f(u_n)\to f(u_0) \text{ in } L^{q+1}(\R^N)
\end{align}
Using (\ref{f5}), (\ref{f6}) and (\ref{f7}), we deduce
\begin{align}
    \lim_{n\to\infty}\int [H(Du_n)^p(1+G(u_n))+ V|f(u_n)|^p]\; dx=\int [H(Du_0)^p(1+G(u_0)) + \int V|f(u_0)|^p]\; dx
\end{align}
 Fatou's lemma ensures that up to a subsequence 
\begin{equation}\label{ff1}
    \begin{split}
        \lim_{n\to\infty}\int H(Du_n)^p\; dx&=\int H(Du_0)^p\;dx\\
     \lim_{n\to\infty} \int V|f(u_n)|^p\; dx&=\int V|f(u_0)|^p\; dx
    \end{split}
\end{equation}
The Brezis-Lieb lemma and (\ref{ff1}) imply $u_n\to u_0$ in $X$.
\item \textbf{Non-Vanishing Case:} If vanishing does not occur then there exists a sequence $\{x^1_n\}\subset \R^N$ such that 
\begin{align}\label{f8}
    \int_{ B_1(0)} |f(\tilde{u}^1_n)|^p\; dx\geq \frac{\rho}{2}
\end{align}
where $\tilde{u}^1_n(x):=u^1_n(x+x^1_n).$
Since the sequence $\{\tilde{u}^1_n\}$ is also bounded in $X$, there exists $w_1\in X$ such that
\begin{align}
    \begin{cases}
    \tilde{u}^1_n\rightharpoonup w_1 \text{ in } X\\
    f(\tilde{u}^1_n)\to f(w_1) \text{ in } L^p(B_1(0)).
\end{cases}
\end{align}
The inequality (\ref{f8}) ensures $w_1\nequiv 0$. Moreover, $\{x^1_n\}$
is unbounded. Our next goal is to show $J'_{\infty,\delta}(w_1)=0$. For $\phi\in C^\infty_c(\R^N)$, one has,
\begin{equation}
    \begin{split}
        \langle J'_{\infty,\delta}(w_1),\phi\rangle&=\lim_{n\to\infty}\langle J'_{\infty,\delta}(\tilde{u}^1_n),\phi\rangle\nonumber\\
    &=\lim_{n\to\infty}\langle J'_\delta(u^1_n),\phi(\cdot-x^1_n)\rangle-\lim_{n\to\infty}\int (V(x+x^1_n)-V(\infty))|f(\tilde{u}^1_n)|^{p-2}f(\tilde{u}^1_n)f'(\tilde{u}^1_n)\phi\; dx
    \end{split}
\end{equation}
By Remark \ref{BL}, we deduce 
$$\langle J'_\delta(u^1_n),\phi \rangle \to 0 \mbox{ uniformly with respect to }\phi.$$
Since $u^1_n\rightharpoonup 0$, one has that $\lim_{n\to\infty}\langle J'_\delta(u^1_n),\phi(\cdot-x^1_n)\rangle=0$ and the condition $(v_1)$ implies $$\lim_{n\to\infty}\int (V(x+x^1_n)-V(\infty))|f(\tilde{u}^1_n)|^{p-2}f(\tilde{u}^1_n)f'(\tilde{u}^1_n)\phi dx=0.$$ Thus, $J'_{\infty,\delta}(w_1)=0$.
Again by Brezis-Lieb lemma, one has,
\begin{equation}\label{f9}
    \begin{split}
        \lim_{n\to\infty}\int [H(Du^1_n)^p-H(Du_n)^p+H(Du_0)^p] dx=0.\\
    \lim_{n\to\infty}\int [|f(u^1_n)|^{q+1}-|f(u_n)|^{q+1}+|f(u_0)|^{q+1}] dx=0.\\
    \lim_{n\to\infty}\int V[|f(u^1_n)|^p-|f(u_n)|^p+|f(u_0)|^p] dx=0.
    \end{split}
\end{equation}
Under the assumption $(v_1)$, we have 
\begin{align}\label{f10}
  \lim_{n\to\infty} \int (V(x)-V(\infty)) |f(u^1_n)|^p dx=0
\end{align}
Using (\ref{f9}) and (\ref{f10}), one can easily conclude,
\begin{equation}\label{f11}
    \begin{split}
        J_\delta(u^1_n)-J_\delta(u_n)+J_\delta(u_0) \to 0 \text{ as }n\to\infty.\\
    J_\delta(u_n)-J_{\infty,\delta}(u^1_n)-J_\delta(u_0)\to 0 \text{ as } n\to\infty.
    \end{split}
\end{equation}
\end{enumerate}
\item Now, we define $$\rho_1=\limsup_{n\to\infty} \sup_{ y\in\R^N}\int_{ B_1(y)} |f(u^2_n)|^p\; dx, \mbox{where}\; u^2_n=u^1_n-w_1(.-x^1_n)\rightharpoonup 0\;\mbox{in}\;X$$\\
If $\rho_1=0$ then by a similar argument as \textbf{Step 3}, we obtain $$||u_n-u_0-w_1(.-x^1_n)||\to 0 \text{ in } X.$$ If $\rho_1\neq 0$ then there exists a sequence $\{x^2_n\}$ such that $\tilde{u}^2_n\rightharpoonup w_2\nequiv 0$. Moreover, $|x^1_n-x^2_n|\to \infty$ as $n\to\infty.$ Arguing as above, we obtain the following:
\begin{equation}\label{f12}
    \begin{split}
        ||H(Du^2_n)||^p_p-||H(Du_n)||^p_p+||H(Du_0)||^p_p+||H(Dw_1(.-x^1_n))||^p_p=o(1)\\
        \int [V|f(u^2_n)|^p-V|f(u_n)|^p+V|f(u_0)|^p
         +V|f(w_1(.-x^1_n))|^p ]\;dx =o(1)\\
        ||f(u^2_n)||_{q+1}^{q+1}-||f(u_n)||_{q+1}^{q+1}+||f(u_0)||_{q+1}^{q+1}+||f(w_1(.-x^1_n))||_{q+1}^{q+1}=o(1)
    \end{split}
\end{equation}
which helps us to obtain
\begin{equation}\label{f13}
    \begin{split}
     J_\delta(u^2_n)=J_\delta(u_n)-J_\delta(u_0)-J_{\infty,\delta}+o(1)\\
        J_{\infty,\delta}(u^2_n)=J_\delta(u^1_n)-J_{\infty,\delta}(w_1)+o(1).   
    \end{split}
\end{equation}
Moreover, using the Brezis-Lieb lemma, we deduce  
 \begin{align}\label{BL1}
      \langle J'_\delta(u^2_n),\phi\rangle=\langle J'_\delta(u_n),\phi\rangle-\langle J'_\delta(u_0),\phi\rangle-\langle J'_{\infty,\delta}(w_1),\phi(\cdot+x^1_n)\rangle+o(1)=o(1)
 \end{align}
and
\begin{align*}
||\tilde{u}^1_n-w_1||^p_{1,p,\R^N}&=||\tilde{u}^1_n||^p_{1,p,\R^N}-||w_1||^p_{1,p,\R^N}+o(1)\\
&=||u_n||^p_{1,p,\R^N}-||u_0||^p_{1,p,\R^N}-||w_1||^p_{1,p,\R^N}+o(1)
\end{align*}
that is, $$||u_n-u_0-w_1(\cdot-x^1_n)||^p_{1,p,\R^N}=||u_n||^p_{1,p,\R^N}-||u_0||^p_{1,p,\R^N}-||w_1(\cdot-x^1_n)||^p_{1,p,\R^N}+o(1).$$
 Since $u^2_n\rightharpoonup 0$, one has $\langle J'_\delta(u^2_n),\phi(\cdot-x^2_n)\rangle\to 0$. Consequently, $J'_{\infty,\delta}(w_2)=0$. 
Using (\ref{f11}) and (\ref{f13}), we have 
\begin{align}
    J_\delta(u_n)=J_\delta(u_0)+J_{\infty,\delta}(u^1_n)+o(1)=J_\delta(u_0)+J_{\infty,\delta}(w_1)+J_{\infty,\delta}(u^2_n)+o(1).
\end{align}
Iterating this process k-times, we obtain $(k-1)$ number of sequences $\{x^j_n\}\subset \R^N$ for $j=1,2,...(k-1)$ and $(k-1)$ number of critical points $w_1,w_2,...w_{k-1}$ of $J'_{\infty,\delta}$ such that
\begin{equation}\label{BB1}
    \begin{split}
        ||u_n-u_0-\sum_{i=1}^{k-1}w_i(.-x^i_n)||^p_{1,p,\R^N}&=||u_n||^p_{1,p,\R^N}-||u_0||^p_{1,p,\R^N}-\sum_{i=1}^{k-1}||w_i(.-x^i_n)||^p_{1,p,\R^N}+o(1)\\
        J_\delta(u_n)&\to J_\delta(u_0)+\sum_{i=1}^{k-1}J_{\infty,\delta}(w_i)+J_{\infty,\delta}(u^k_n)
    \end{split}
\end{equation}
where  $u^k_n;=u_n-u_0-\sum_{i=1}^{k-1}w_i(\cdot-x^i_n)\rightharpoonup 0$ in $X$.\\
\item Since $J'_{\infty,\delta}(w_i)=0$, by property (ii) in Lemma \ref{NA} there exists a constant $C>0$ such that $||w_i||\geq C.$ Using this fact together with (\ref{BB1}), we can conclude that the iteration stops after some finite index $k\in \N$.
\end{enumerate}
\end{proof}
\begin{lem}\label{ML}
    Suppose that $(v_1)$ and $(v_2)$ hold, $N\geq 3$,  $(\alpha,p)\in D^N$ and $2\alpha p-1\leq q<2\alpha p^*-1$. For $\delta\in I$, let $\{v_n\}$ be a bounded $(PS)_{C_\delta}$ sequence of $J_\delta$. Then there exists $v_\delta\in X$ such that $J'_\delta(v_\delta)=0$ and $J_\delta(v_\delta)=C_\delta$, where $C_\delta$ is defined by (\ref{Cdef}).
\end{lem}
\begin{proof}
    By using lemma \ref{GCL}, there exist $v_\delta\in X$, $k\in \N\cup \{0\}$ and $\{w_1,w_2,...w_k\}\subset X$ such that
    \begin{enumerate}[label=(\roman*)]
        \item $v_n\rightharpoonup v_\delta$,\; $J'_\delta(v_\delta)=0$ and $J_\delta(v_\delta)\geq 0$.
        \item $w_i\nequiv 0$ and $J_{\infty,\delta}'(w_i)=0$, for $1\leq i\leq k$.
        \item $J_\delta(v_n)\to J_\delta(v_\delta)+\sum_{i=1}^{k} J_{\infty,\delta}(w_i)$.
    \end{enumerate}
Clearly, $ J_{\infty,\delta}(w_i)\geq m_{\infty,\delta}$. If $k\neq 0$ then $C_\delta\geq m_{\infty,\delta}$, which contradicts the fact that $C_\delta<m_{\infty,\delta}$. Hence, $k=0$. By using lemma \ref{GCL}, we have $v_n\to v_\delta$ in $X$ and $J_\delta(v_\delta)=C_\delta$.
\end{proof}

\begin{cor}\label{cor2}
    If all the assumptions of lemma \ref{ML} are satisfied. Then for almost every $\delta\in I$, there exists $v_\delta\in X$ such that $J_{\delta}(v_\delta)=C_{\delta}$ and $J'_{\delta}(v_\delta)=0$.
\end{cor}
\begin{proof}
     Theorem \ref{MP} ensures us that for almost every $\delta\in I$, $J_\delta$ has a bounded $(PS)_{C_\delta}$ sequence. Hence by using the above lemma, we get the result.
\end{proof}

\section{Proof of Theorem \ref{TC}}\label{TCP}

Let $l=\inf\{J(v): v\in M\}(>0)$ and $\{v_n\}\subset M$ be a minimizing sequence. Also, $$NJ(v_n)=NJ(v_n)-P(v_n)=\int H(Dv_n)^p dx$$ So, $\{H(Dv_n)\}$ is bounded in $L^p(\R^n)$. We will prove $\{v_n\}$ is bounded in $X$. From (\ref{ineq1}), we have 
\begin{align}\label{AA1}
    \int |(f(v_n))|^{q+1} dx&\leq  C [\int |f(v_n)|^p dx]^\frac{\gamma(q+1)}{p} [\int H(Dv_n)^p dx]^{\frac{p^*}{p}(1-\frac{\gamma(q+1)}{p})}\nonumber\\
    &\leq \epsilon \int |f(v_n)|^p + C_\epsilon (\int H(Dv_n)^p dx)^\frac{p^*}{p} 
\end{align}
where $\epsilon>0$ and $\gamma=\frac{p(2\alpha p^*-q-1)}{(q+1)(2\alpha p^*-p)}.$ Now, by using the Poho$\check{z}$aev identity and (\ref{AA1}), we get
\begin{align*}
    \frac{N}{p}\int V|f(v)|^p dx&=\frac{\lambda N}{q+1}\int |f(v)|^{q+1} dx -\frac{N-p}{p}\int H(Dv)^p dx.\\
    &\leq \frac{\lambda N\epsilon}{q+1} \int |f(v_n)|^p + \frac{\lambda NC_\epsilon}{q+1} (\int H(Dv_n)^p dx)^\frac{p^*}{p} -\frac{N-p}{p}\int H(Dv)^p dx.
\end{align*}

Choose $\epsilon=\frac{V(q+1)}{2p\lambda}$ and by using the fact that $\{H(Dv_n\})$ is bounded in $L^p(\R^n)$, we can conclude that $\{\int V|f(v_n)|^p\}$ is bounded in $\R$. Finally, (\ref{n2}) ensures the boundedness of $\{v_n\}$ in $X$. By lemma \ref{Comp}, up to a subsequence $v_n\rightharpoonup v$ in X and $f(v_n)\rightarrow f(v)$ in $L^{q+1}(\R^n)$. 
Now we will show $v\in M$ and $l=J(v)$. Now,
\begin{align}
     P(v_n)=\frac{N-p}{p}\int H(Dv_n)^p dx+ \frac{N}{p}\int V|f(v_n)|^p dx-\frac{\lambda N}{q+1}\int |f(v_n)|^{q+1} dx=0.
\end{align}
For simplicity, let 
\begin{enumerate}
    \item $a_n=\int H(Dv_n)^p\; dx$, $a=\lim_{n\to\infty}a_n$ and $\bar{a}=\int H(Dv)^p \;dx$.
    \item $b_n=\int V|f(v_n)|^p\; dx$, $b=\lim_{n\to\infty}b_n$ and $\bar{b}=\int V|f(v)|^p\; dx$.
    \item $c_n=\int |f(v_n)|^{q+1}\; dx$, $c=\lim_{n\to\infty}c_n$ and $\bar{c}=\int |f(v)|^{q+1} \;dx$. 
\end{enumerate}
Clearly, $\bar{a}\leq a$, $\bar{b}\leq b$ and $c=\bar{c}$. Our claim is $a=\bar{a}$ and $ b=\bar{b}$. For the time being, let us assume that the claim is true. Now, 
\begin{equation}
    P(v)=\frac{N-p}{p}\int H(Dv)^p dx+ \frac{N}{p}\int V|f(v)|^p dx-\frac{\lambda N}{q+1}\int |f(v)|^{q+1} dx=\lim_{n\to \infty} P(v_n)=0
\end{equation}
and 
\begin{align}
    J(v)=\lim_{n\to\infty} J(v_n)=l>0
\end{align}
So, $v\in X_r\setminus\{0\}$. Hence, $v\in M$ and $J(v)=\inf_{u\in M} J(u)$. Moreover, by (iv) in lemma \ref{NA}, we have $J'(v)=0$. Without loss of generality, we can assume $v$ is non-negative and \cite[ \text{ Proposition 4.3 }]{CM} ensures $v\in C^1(\R^N)$. The function $u=f(v)\in C^1(\R^N)$ is a non-trivial non-negative bounded ground state solution of (\ref{maineq}).\\
Now, we will prove our claim. If the claim is not true then $\bar{a}+\bar{b}<a+b$. Consider the following equations
\begin{align*}
    \frac{1}{p}a+\frac{1}{p}b-\frac{\lambda}{q+1}c=l
\end{align*}
and 
\begin{align*}
    \frac{N-p}{p}a+\frac{N}{p}b-\frac{\lambda N}{q+1}c=0
\end{align*}
Clearly, $c\neq 0$. Define two functions $g_1,g_2:(0, \infty)\to \R$ by 
\begin{align*}
    g_1(t)=\frac{1}{p}\bar{a}t^{N-p}+\frac{1}{p}\bar{b}t^{N}-\frac{\lambda}{q+1}\bar{c}t^{N}
\end{align*}
and 
\begin{align*}
    g_2(t)=\frac{1}{p}at^{N-p}+\frac{1}{p}bt^{N}-\frac{\lambda}{q+1}ct^{N}
\end{align*}
 It is clear that $g_1(t)<g_2(t)$, for all $t>0$. Also, $g_2'(1)=0$ and $g_2(1)=l$. Hence there exists $t_0>0$ such that $g_1(t_0)=\max_{t>0} g_1(t)<l$. Now, consider the function $v_{t_0}(x)=v(\frac{x}{t_0})$, which satisfies $J(v_{t_0})=g_1(t_0)<l$ and $P(v_{t_0})=t_0g_1'(t_0)=0$. Hence, $v_{t_0}\in M$ and $ J(v_{t_0})<l$, which is a contradiction.

\section{Proof of Theorem \ref{TV}}\label{TVP}
Now we are ready to prove our main theorem which we split into two steps:
\begin{enumerate}[label= Step \arabic*:]
\item In this step, our aim is to show the existence of a non-trivial critical point of the functional $J$. By corollary \ref{cor2}, we are allowed to choose a sequence $\delta_n\nearrow 1$ such that for any $n\geq 1$, there exists $v_n\in X\setminus\{0\}$ satisfying
\begin{equation}\label{L1}
    J_{\delta_n}(v_n)=\frac{1}{p}\int H(Dv_n)^p+ V(x)|f(v_n)|^p]dx-\frac{\lambda\delta_n}{q+1}\int |f(v_n)|^{q+1} dx=C_{\delta_n}
\end{equation}
and
\begin{equation*}
        J'_{\delta_n}(v_n)=0.
\end{equation*}
By using lemma \ref{ET} and $\langle J'_{\delta_n}(v_n),\frac{f(v_n)}{f'(v_n)}\rangle=0$, we deduce
\begin{align}\label{L2}
    \int H(Dv_n)^p(2\alpha-F(v_n)) dx +\int V(x)|f(v_n)|^p dx-\lambda\delta_n\int |f(v_n)|^{q+1} dx=0
\end{align}
where $F(v_n)=\frac{2\alpha-1}{1+(2\alpha)^{p-1}|f(v_n)|^{p(2\alpha-1)}}$.
Moreover, $v_n$ satisfies the following Poho$\check{z}$aev identity,
\begin{align}\label{L3}
    \frac{N-p}{p}\int H(Dv_n)^p dx+ \frac{N}{p}\int V(x)|f(v_n)|^p dx&+\frac{1}{p}\int \langle\nabla V(x), x\rangle |f(v_n)|^p dx\nonumber\\
    &-\frac{N\lambda \delta_n}{q+1}\int |f(v_n)|^{q+1} dx=0.
\end{align}
 Multiplying (\ref{L1}) and (\ref{L3}) by $N$ and $r=\frac{q-2\alpha p+1}{2\alpha p}$ respectively and then adding those results we get 
 \begin{align}\label{L4}
      [\frac{N}{p}+\frac{r(N-p)}{p}]\int H(Dv_n)^p dx+ \frac{N}{p}\int V(x)|f(v_n)|^p dx + \frac{r}{p}\int [NV(x)+\langle\nabla V(x), x\rangle]|f(v_n)|^p dx\nonumber\\=\frac{N\lambda\delta_n}{2\alpha p}\int |f(v_n)|^{q+1} dx+NC_{\delta_n}
 \end{align}
From (\ref{L2}) and (\ref{L4}), we deduce
\begin{equation}\label{L5}
\begin{split}
    \frac{r(N-p)}{p}\int H(Dv_n)^p dx+ \frac{N(2\alpha-1)}{2\alpha p}\int V(x)|f(v_n)|^p dx + \frac{r}{p}\int [NV(x)+\langle\nabla V(x), x\rangle]|f(v_n)|^p dx\\ \nonumber+\frac{N}{2\alpha p}\int H(Dv_n)^pF(v_n) dx=NC_{\delta_n}
    \end{split}
\end{equation}
Since $V$ satisfies $(v_2)$, $(\alpha,p)\in D_N$ and $\{C_{\delta_n}\}$ is bounded so (\ref{L5}) ensures the boundedness of $$\{\int H(Dv_n)^p dx+\int V(x)|f(v_n)|^p dx\}_n.$$ Hence $\{v_n\}$ is bounded in $X$. Now,
\begin{equation}\label{L6}
    J(v_n)=J_{\delta_n}(v_n)+\frac{\lambda(\delta_n-1)}{q+1}\int |f(v_n)|^{q+1} dx
    =C_{\delta_n}+\frac{\lambda(\delta_n-1)}{q+1}\int |f(v_n)|^{q+1} dx
\end{equation}
and, 
\begin{equation*}
\begin{split}
    \langle J'(v_n),w\rangle&= \langle J_{\delta_n}(v_n),w\rangle+\frac{\lambda(\delta_n-1)}{q+1}\int |f(v_n)|^{q-1}f(v_n)f'(v_n)w dx\\
    &=\frac{\lambda(\delta_n-1)}{q+1}\int |f(v_n)|^{q-1}f(v_n)f'(v_n)w\;dx
    \end{split}
\end{equation*}
that is, 
\begin{align}\label{L7}
    J'(v_n)=J'_{\delta_n}(v_n)+\frac{\lambda (\delta_n-1)}{q+1} g(v_n)
\end{align}
where $g(v_n)=|f(v_n)|^{q-1}f(v_n)f'(v_n)\in X'$. Since $\{v_n\}$ is bounded in $X$, by Banach-Steinhaus theorem we have $\{g(v_n)\}$ is bounded in $X'$. Using (\ref{L6}), (\ref{L7}), and the left continuity of the map $\delta\to C_\delta$, we obtain
\begin{equation*}
    J(v_n)\to C_1 \text{ as } n\to \infty.
\end{equation*}
and,
\begin{equation*}
    J'(v_n)\to 0 \text{ as } n\to \infty.
\end{equation*}
Hence $\{v_n\}$ is a bounded $(PS)_{C_1}$ sequence for $J$. By the lemma \ref{ML}, there exists $\tilde{v}\in X$ such that $J(\tilde{v})=C_1$ and $J'(\tilde{v})=0$.
\item Let $E=\{ v\in X\setminus\{0\} : J'(v)=0\}$ and $S=\inf_{v\in E} J(v)$. Clearly, $E$ is nonempty and $0\leq S\leq C_1<m_{\infty,1}$. Let $\{v_n\}\subset E$ be a minimizing sequence. Therefore, $J(v_n)\to S$ as $n\to\infty$ and $J'(v_n)=0$, for all $n\in \N$. Using a similar argument as \textbf{Step I}, we can prove that $\{v_n\}$ is a bounded $(PS)_S$ sequence for $J$. Using the argument introduced in the proof of the lemma \ref{SL}, there exists $v_0\in X$ such that $v_n\to v_0$ and $J(v_0)=S$. Without loss of generality, we can assume $v_0$ is nonnegative. We want to prove $v_0\nequiv 0$.\\ Define a map $T: X\to \R$ as 
\begin{equation*}
    T(v):= \int [H(Dv)^p+V(x)|f(v)|^p] dx.
\end{equation*}
which is continuous. If $v_n\to 0$ in $X$ then $T(v_n)\to 0$.\\
Now,
\begin{equation}\label{ineq2}
\begin{split}
    \langle J'(v_n),\frac{f(v_n)}{f'(v_n)}\rangle&= \int H(Dv_n)^p(1+G(v_n)) dx+ \int V(x)|f(v_n)|^p dx-\lambda\int |f(v_n)|^{q+1} dx\\
    &\geq  \int H(Dv_n)^p dx+ \int V(x)|f(v_n)|^p dx-\lambda\int |f(v_n)|^{q+1} dx
    \end{split}
\end{equation}
where $G(v_n)=\frac{(2\alpha-1)(2\alpha)^{p-1}|f(v_n)|^{p(2\alpha-1)}}{1+(2\alpha)^{p-1}|f(v_n)|^{p(2\alpha-1)}}\geq 0.$
Let $$v_n\in S(\beta_n)=\{v \in X: \int [H(Dv)^p+V(x)|f(v)|^p] dx=\beta_n^p\}$$ Using (\ref{ineq1}), (\ref{ineq2}) and $J'(v_n)=0$, we deduce 
\begin{equation*}
\begin{split}
     \beta_n^p=\int H(Dv_n)^p dx+ \int V(x)|f(v_n)|^p dx&\leq \langle J'(v_n),\frac{f(v_n)}{f'(v_n)}\rangle +\lambda\int |f(v_n)|^{q+1} dx\\
     &\leq C\beta_n^m, \text{ where } m>p
     \end{split}
\end{equation*}
Hence, the sequence $\{\beta_n\}$ is bounded below by some positive constant, which contradicts the fact that $T(v_n)\to 0$. Hence, $\tilde{v}=|v_0|$ is a non-trivial non-negative ground state solution of (\ref{maineq2}). By using lemma \ref{Bdd} and \cite[\text{ Proposition 4.3}]{CM}, we have $\tilde{v}$ is bounded and in $C^1(\R^N)$. 
 Thus, $u_0=f(\tilde{v})\in C^1(\R^N)$ is a non-trivial non-negative bounded ground state solution of (\ref{maineq}).
\end{enumerate}
 
\section{Acknowledgement}
We would like to thank Prof. Adimurthi for his invaluable advice and assistance. The first author was supported by MATRICS project no MTR/2020/000594.
\nocite{*}
\bibliographystyle{plain}
\bibliography{Reference.bib}

\begin{thebibliography}{10}

\bibitem{AC}
Claudianor~O. Alves.
\newblock Existence of positive solutions for a problem with lack of
  compactness involving the {$p$}-{L}aplacian.
\newblock {\em Nonlinear Anal.}, 51(7):1187--1206, 2002.

\bibitem{AF}
Claudianor~O. Alves and Giovany~M. Figueiredo.
\newblock Multiple solutions for a quasilinear {S}chr\"{o}dinger equation on
  {$\Bbb{R}^N$}.
\newblock {\em Acta Appl. Math.}, 136:91--117, 2015.

\bibitem{AA}
A.~Azzollini and A.~Pomponio.
\newblock On the {S}chr\"{o}dinger equation in {$\Bbb R^N$} under the effect of
  a general nonlinear term.
\newblock {\em Indiana Univ. Math. J.}, 58(3):1361--1378, 2009.

\bibitem{KB}
Kaushik Bal, Prashanta Garain, and Tuhina Mukherjee.
\newblock On an anisotropic {$p$}-{L}aplace equation with variable singular
  exponent.
\newblock {\em Adv. Differential Equations}, 26(11-12):535--562, 2021.

\bibitem{HP}
H.~Berestycki and P.-L. Lions.
\newblock Nonlinear scalar field equations. {I}. {E}xistence of a ground state.
\newblock {\em Arch. Rational Mech. Anal.}, 82(4):313--345, 1983.

\bibitem{BG}
AV~Borovskii and AL~Galkin.
\newblock Dynamic modulation of an ultrashort high-intensity laser pulse in
  matter.
\newblock {\em JETP}, 77(4):562--573, 1993.

\bibitem{BL}
Ha\"{\i}m Br\'{e}zis and Elliott Lieb.
\newblock A relation between pointwise convergence of functions and convergence
  of functionals.
\newblock {\em Proc. Amer. Math. Soc.}, 88(3):486--490, 1983.

\bibitem{ZJ}
Jianqing Chen and Qian Zhang.
\newblock Ground state solution of pohožaev type for quasilinear schrödinger
  equation involving critical exponent in orlicz space.
\newblock {\em Mathematics}, 7(9), 2019.

\bibitem{MC}
M\'{o}nica Clapp and Luis Lopez~Rios.
\newblock Entire nodal solutions to the pure critical exponent problem for the
  {$p$}-{L}aplacian.
\newblock {\em J. Differential Equations}, 265(3):891--905, 2018.

\bibitem{CM}
Matteo Cozzi, Alberto Farina, and Enrico Valdinoci.
\newblock Monotonicity formulae and classification results for singular,
  degenerate, anisotropic {PDE}s.
\newblock {\em Adv. Math.}, 293:343--381, 2016.

\bibitem{SEV}
Jo\~{a}o~Marcos do~\'{O} and Uberlandio Severo.
\newblock Solitary waves for a class of quasilinear {S}chr\"{o}dinger equations
  in dimension two.
\newblock {\em Calc. Var. Partial Differential Equations}, 38(3-4):275--315,
  2010.

\bibitem{FC}
Csaba Farkas and Patrick Winkert.
\newblock An existence result for singular {F}insler double phase problems.
\newblock {\em J. Differential Equations}, 286:455--473, 2021.

\bibitem{AK}
Andreas Floer and Alan Weinstein.
\newblock Nonspreading wave packets for the cubic {S}chr\"{o}dinger equation
  with a bounded potential.
\newblock {\em J. Funct. Anal.}, 69(3):397--408, 1986.

\bibitem{Jean}
Louis Jeanjean.
\newblock Local conditions insuring bifurcation from the continuous spectrum.
\newblock {\em Math. Z.}, 232(4):651--664, 1999.

\bibitem{JJ}
Louis Jeanjean.
\newblock On the existence of bounded {P}alais-{S}male sequences and
  application to a {L}andesman-{L}azer-type problem set on {${\bf R}^N$}.
\newblock {\em Proc. Roy. Soc. Edinburgh Sect. A}, 129(4):787--809, 1999.

\bibitem{KJ}
Louis Jeanjean and Kazunaga Tanaka.
\newblock A positive solution for a nonlinear {S}chr\"{o}dinger equation on
  {$\Bbb R^N$}.
\newblock {\em Indiana Univ. Math. J.}, 54(2):443--464, 2005.

\bibitem{Kur}
Susumu Kurihara.
\newblock Exact soliton solution for superfluid film dynamics.
\newblock {\em J. Phys. Soc. Japan}, 50(11):3801--3805, 1981.

\bibitem{PLL}
Pierre-Louis Lions.
\newblock Sym\'{e}trie et compacit\'{e} dans les espaces de {S}obolev.
\newblock {\em J. Functional Analysis}, 49(3):315--334, 1982.

\bibitem{Wang2}
Jia-quan Liu, Ya-qi Wang, and Zhi-Qiang Wang.
\newblock Soliton solutions for quasilinear {S}chr\"{o}dinger equations. {II}.
\newblock {\em J. Differential Equations}, 187(2):473--493, 2003.

\bibitem{Liu}
Jia-quan Liu, Ya-qi Wang, and Zhi-Qiang Wang.
\newblock Solutions for quasilinear {S}chr\"{o}dinger equations via the
  {N}ehari method.
\newblock {\em Comm. Partial Differential Equations}, 29(5-6):879--901, 2004.

\bibitem{Wang1}
Jiaquan Liu and Zhi-Qiang Wang.
\newblock Soliton solutions for quasilinear {S}chr\"{o}dinger equations. {I}.
\newblock {\em Proc. Amer. Math. Soc.}, 131(2):441--448, 2003.

\bibitem{QW}
Xiangqing Liu, Jiaquan Liu, and Zhi-Qiang Wang.
\newblock Ground states for quasilinear {S}chr\"{o}dinger equations with
  critical growth.
\newblock {\em Calc. Var. Partial Differential Equations}, 46(3-4):641--669,
  2013.

\bibitem{MW}
Carlo Mercuri and Michel Willem.
\newblock A global compactness result for the {$p$}-{L}aplacian involving
  critical nonlinearities.
\newblock {\em Discrete Contin. Dyn. Syst.}, 28(2):469--493, 2010.

\bibitem{VM}
I-I Mezei and O~Vas.
\newblock Existence results for some dirichlet problems involving
  finsler--laplacian operator.
\newblock {\em Acta Mathematica Hungarica}, 157(1):39--53, 2019.

\bibitem{LS}
Luigi Montoro and Berardino Sciunzi.
\newblock Pohozaev identity for {F}insler anisotropic problems.
\newblock {\em NoDEA Nonlinear Differential Equations Appl.}, 30(3):Paper No.
  33, 15, 2023.

\bibitem{PS}
Markus Poppenberg, Klaus Schmitt, and Zhi-Qiang Wang.
\newblock On the existence of soliton solutions to quasilinear
  {S}chr\"{o}dinger equations.
\newblock {\em Calc. Var. Partial Differential Equations}, 14(3):329--344,
  2002.

\bibitem{MR}
M.~M. Rao and Z.~D. Ren.
\newblock {\em Theory of {O}rlicz spaces}, volume 146 of {\em Monographs and
  Textbooks in Pure and Applied Mathematics}.
\newblock Marcel Dekker, Inc., New York, 1991.

\bibitem{BR}
Burke Ritchie.
\newblock Relativistic self-focusing and channel formation in laser-plasma
  interactions.
\newblock {\em Phys. Rev. E}, 50:R687--R689, Aug 1994.

\bibitem{FG}
Ghulamullah Saeedi and Farhad Waseel.
\newblock Existence of solutions for a class of quasilinear elliptic equations
  involving the p-laplacian.
\newblock {\em Complex Variables and Elliptic Equations}, 0(0):1--25, 2022.

\bibitem{US}
Uberlandio Severo.
\newblock Existence of weak solutions for quasilinear elliptic equations
  involving the {$p$}-{L}aplacian.
\newblock {\em Electron. J. Differential Equations}, pages No. 56, 16, 2008.

\bibitem{YW}
Youjun Wang and Wenming Zou.
\newblock Bound states to critical quasilinear {S}chr\"{o}dinger equations.
\newblock {\em NoDEA Nonlinear Differential Equations Appl.}, 19(1):19--47,
  2012.

\bibitem{CX}
C.~Xia.
\newblock On a class of anisotropic problems.
\newblock 2012.

\bibitem{Chen}
Liping Xu and Haibo Chen.
\newblock Ground state solutions for quasilinear {S}chr\"{o}dinger equations
  via {P}oho\v{z}aev manifold in {O}rlicz space.
\newblock {\em J. Differential Equations}, 265(9):4417--4441, 2018.

\bibitem{SM}
Minbo Yang, Carlos~Alberto Santos, and Jiazheng Zhou.
\newblock Least energy nodal solutions for a defocusing {S}chr\"{o}dinger
  equation with supercritical exponent.
\newblock {\em Proc. Edinb. Math. Soc. (2)}, 62(1):1--23, 2019.

\end{thebibliography}
\end{document}